\documentclass{amsart}[11pt]

\usepackage{amscd}

\usepackage{amsmath}
\usepackage{amsfonts}
\usepackage{amssymb}
\usepackage{graphicx}
\usepackage{amsthm}%
\providecommand{\U}[1]{\protect\rule{.1in}{.1in}}
\newtheorem{theorem}{Theorem}

\newtheorem{corollary}[theorem]{Corollary}

\newtheorem{definition}[theorem]{Definition}
\newtheorem{example}[theorem]{Example}

\newtheorem{proposition}[theorem]{Proposition}
\newtheorem{remark}[theorem]{Remark}

\newcommand{\K}{\mathbb K}

\newcommand{\g}{\mathfrak{g}}

\newcommand{\sll}{\mathfrak{sl}_2(\mathbb{K})}

\begin{document}

\title[ On classification of Filiform Hom-Lie algebras]{  On classification of Filiform Hom-Lie algebras}

\author{ Abdenacer Makhlouf and Mourad Mehidi }%

\address{ Universit\'{e} de Haute Alsace,  Laboratoire de Math\'{e}matiques, Informatique et Applications,
4, rue des Fr\`{e}res Lumi\`{e}re F-68093 Mulhouse, France}%
\email{Abdenacer.Makhlouf@uha.fr, Mourad.Mehidi@uha.fr}

\maketitle
\begin{abstract}
The purpose of this paper is to introduce and study nilpotent and filiform Hom-Lie algebras. Moreover, we extend  Vergne and Khakimdjanov's approach to Hom-type algebras  and provide a classification of filiform Hom-Lie algebras of
dimension $n,n\leq7$.
\end{abstract}

\section*{Introduction}

Hom-type algebra structures
appeared  first in quasi-deformation of Lie algebras of vector fields.
Discrete modifications of vector fields via twisted derivations lead
to Hom-Lie and quasi-Hom-Lie structures in which the Jacobi
condition is twisted. The first examples of $q$-deformations, in which the derivations are replaced by $\sigma$-derivations,
concerned the Witt and Virasoro algebras, see for example
\cite{AizawaSaito}. A  general study and
construction of Hom-Lie algebras  and a more general framework bordering color and
Lie superalgebras were considered
 in \cite{HLS,LS}. In the subclass of Hom-Lie
algebras skew-symmetry is untwisted, whereas the Jacobi identity is
twisted by a single linear map, reducing to ordinary Lie algebras when the twisting linear
map is the identity map.
The notion of Hom-associative algebras generalizing associative
algebras to a situation where associativity bracket  is twisted by a
linear map was introduced  in \cite{MS}. It turns out  that the
commutator bracket multiplication defined using the multiplication
of a Hom-associative algebra leads naturally to a Hom-Lie algebra.
The enveloping
algebras of Hom-Lie algebras were discussed in \cite{Yau:EnvLieAlg} and the graded case in \cite{AM10}.
The fundamentals of the formal deformation theory and associated
cohomology structures for Hom-Lie algebras have been considered
initially  in \cite{AM 2007} and completed in \cite{AEM11} and simultaneously in \cite{Sheng}. 

Nilpotent Lie algebras play an important role in the study  of Lie algebras structure. Classification of nilpotent Lie algebras are of great interest since Umlauf's thesis in 1891, where he classified filiform Lie algebras up to dimension 7. It turns out that this classification included some mistakes and was  incomplete. Since then it was a big challenge to classify nilpotent Lie algebras or the special class of filiform Lie algebras. 
Graded nilpotent Lie algebras were classified by Vergne \cite{vergne}. She also showed that there exists a well-adapted cohomology permitting to extend the classical results of Gerstenhaber-Nijenhuis-Richardson concerning deformations to the scheme $\mathcal{N}^n$ defined by Jacobi's identities and nilpotent relations. She studied the irreducible components meeting the open set of filiform brackets.
Then, in \cite{Hakim1,Hakim2}  Khakimdjanov  described all irreducible components of $\mathcal{N}^n$.  Moreover,    this approach based on   deformations of  the Lie algebra $L_n$ into a filiform Lie algebra with appropriate cocycles, was used to provide explicit  classifications of filiform Lie algebras, see for example   \cite{4,7,8,9}. See also for classification problem  \cite{1,2,3,5,12,13,16}.

In this paper, we aim to study nilpotent Hom-Lie algebras and, in particular,  the subclass of filiform Hom-Lie algebras. We introduce the definitions and discuss their properties. Moreover, following the Vergne-Khakimdjanov's  approach, we provide the classification of filiform Hom-Lie algebras up to dimension 7. In the first section, we review the basics on Hom-Lie algebras and cohomology. In Section 2, we introduce the definitions and some properties of solvable, nilpotent and filiform Hom-Lie algebras. Section 3, extends the Yau twisting  to nilpotent Hom-Lie algebras 
along algebra morphisms and provides various examples. From Section 4, we deal with classification of filiform Hom-Lie algebras. We extend Vergne-Khakimdjanov approach to Hom-type algebras, using infinitesimal deformations of the model filiform Hom-Lie algebra. We discuss the adapted basis changes that lead to a classification. Classification of $m$-dimensional filiform Hom-Lie algebras for $m\leq 5$  is given in Section 5 and  dimension 6 (resp. dimension 7)  is provided in Section 6 (resp. Section 7). In the last section, we come out with the classification of multiplicative filiform Hom-Lie algebras. 
\section{Generalities on Hom-Lie algebras}

In this section, we review the basics on Hom-Lie algebras and  the Hom-Type Chevalley-Eilenberg cohomology.  In the sequel, all vector spaces and algebras are considered over  $\mathbb{K}$,  an algebraically  closed field of characteristic 0, even if many results could be stated in a more general situation. 

\subsection{Definitions and properties}
\begin{definition} \label{def:HomLie}
A \emph{Hom-Lie algebra} is a triple $(\mathfrak{g}, [\cdot ,
\cdot], \alpha)$ consisting of a vector space $\mathfrak{g}$ on which  $ [\cdot ,
\cdot],: \mathfrak{g}\times \mathfrak{g} \rightarrow \mathfrak{g}$ is
a bilinear map and $\alpha: \mathfrak{g} \rightarrow \mathfrak{g}$
 a linear map
 satisfying
\begin{eqnarray} & [x,y]=-[y,x],
\quad {\text{(skew-symmetry)}} \\ \label{HomJacobiCondition} &
\circlearrowleft_{x,y,z}{[\alpha(x),[y,z]]}=0 \quad
{\text{(Hom-Jacobi identity)}}
\end{eqnarray}
for all $x, y, z$ in $\mathfrak{g}$ and  where $\circlearrowleft_{x,y,z}$ denotes
summation over the cyclic permutation on $x,y,z$.

The Hom-Lie algebra is said to be multiplicative if, in addition,  we have
$$\alpha ([x,y])=[\alpha(x),\alpha (y)]  \text { for all } x,y\in \mathfrak{g}.$$

A Hom-Lie algebra $(\mathfrak{g}, [\cdot ,
\cdot], \alpha)$ is said of Lie  type if there exist a Lie bracket $ [\cdot ,
\cdot]'$ on $\mathfrak{g}$ such that $ [\cdot ,
\cdot]= \alpha \circ  [\cdot ,
\cdot]'$.
\end{definition}
The linear map $\alpha$ is called twist map or structure map.  We recover classical Lie algebra when $\alpha =id_\mathfrak{g}$ and the identity \eqref{HomJacobiCondition} is the Jacobi identity in this case.\\ 

\begin{definition}
Let  $\left( \mathfrak{g},[\cdot,\cdot ],\alpha \right) $ be a Hom-Lie algebra. \\
$\bullet $ A subspace  $\mathfrak{h}$ of $\mathfrak{g}$ is said to be a \emph{Hom-Lie subalgebra} if for $x\in \mathfrak{h}$ and $y\in \mathfrak{g}$,  we have $[x,y]\in \mathfrak{h}$ and $\alpha (x)\in \mathfrak{h}.$

\noindent $\bullet$ A subspace  $I$ of $\mathfrak{g}$ is said to be an \emph{ideal} if for $x\in I$ and $y\in \mathfrak{g}$, we have $[x,y]\in I$ and $\alpha (x)\in I$, that is $\left[  \mathfrak{g},I\right]  \subset I\ $and $\alpha \left(  I\right)  \subset
I$.

\noindent $\bullet$ A Hom-Lie algebra in which the commutator is not identically zero and which has no proper ideals is called \emph{simple}.

\end{definition}

Let $\left( \mathfrak{g}, [\cdot ,
\cdot],,\alpha \right) $ and
$\mathfrak{g}^{\prime }=\left(\mathfrak{g}^{\prime }, [\cdot ,
\cdot]^{\prime
},\alpha^{\prime }\right) $ be two Hom-Lie algebras. \\ A linear map
$f :\mathfrak{g}\rightarrow \mathfrak{g}^{\prime }$ is
a \emph{ Hom-Lie algebras morphism} if%
$$
 [\cdot ,
\cdot] ^{\prime }\circ (f\otimes f)=f\circ [\cdot ,
\cdot] \quad \text{
and } \qquad f\circ \alpha=\alpha^{\prime }\circ f.
$$
It is called a \emph{weak morphism}  if  only the first condition holds.\\
In particular, Hom-Lie algebras $\left(\mathfrak{g}, [\cdot ,
\cdot],\alpha \right) $ and
$\left(\mathfrak{g},[\cdot ,
\cdot]^{\prime },\alpha^{\prime }\right) $ are isomorphic if
there exists a
bijective linear map $f\ $such that%
$$
[\cdot ,
\cdot] =f^{-1}\circ[\cdot ,
\cdot] ^{\prime }\circ (f\otimes f)\qquad
\text{ and }\qquad \alpha= f^{-1}\circ \alpha^{\prime }\circ
f.
$$

\begin{example}[Hom-Type $\sll$] The triple 
$(\sll ,[\cdot,\cdot],\alpha)$ defines a $3$-dimensional  Hom-Lie algebra generated by  $$ H=\left(
\begin{array}{cc}
1 & 0 \\
0 & -1 \\
\end{array}
\right),
\\E=\left(
  \begin{array}{cc}
  0 & 0 \\
    1 & 0 \\
     \end{array}
      \right)
      ,F=\left(
      \begin{array}{cc}
       0 & 1 \\
          0 & 0 \\
           \end{array}
           \right),
$$
with $[A,B]=AB-BA$, that is 
$$[H,E]=-2 E,\ [H,F]=2 F,\  [E,F]=-H,
$$
 and where the twist maps are given with respect to the basis by the matrices
$$
\mathcal{M}_\alpha=\left(
\begin{array}{ccc}
a & c & d \\
2d & b & e \\
2c & f & b \\
\end{array}
\right),\; \hbox{where}\; a,b,c,d,e,f \in \mathbb{K}.$$

 Now, if we consider the $q$-deformation of $\sll$, that is viewed in terms differential operators and where the usual derivation is replaced by a $\sigma$-derivation (Jackson derivative), we obtain the following Hom-Lie algebra
defined with respect to a basis $\{x_{1},x_{2},x_{3}\}$ by 
\begin{align*}
& [x_{1},x_{2}]=-2qx_{2}, &  \alpha
(x_{1})=qx_{1},\\
&[x_{1},x_{3}]=2x_{3}, &  \alpha
(x_{2})=q^{2}x_{2},\\
& [x_{2},x_{3}]=-\frac{1}{2}(1+q)x_{1}, &  \alpha
(x_{3})=qx_{3},
\end{align*}
where $q$ is a parameter in $\mathbb{K}$.  If $q=1$,  we recover the classical
$\sll$.

\end{example}
\begin{example}\label{Example1HLie}
Let $\{x_1,x_2,x_3\}$  be a basis of a $3$-dimensional vector space
$\mathfrak{g}$ over $\K$. The following bracket and   linear map $\alpha$ on
$\g=\K^3$ define a Hom-Lie algebra over $\K${\rm :}
$$
\begin{array}{cc}
\begin{array}{ccc}
 [ x_1, x_2 ] &= &a x_1 +b x_3 \\ {}
 [x_1, x_3 ]&=& c x_2  \\ {}
 [ x_2,x_3 ] & = & d x_1+2 a x_3,
 \end{array}
 & \quad
  \begin{array}{ccc}
  \alpha (x_1)&=&x_1 \\
 \alpha (x_2)&=&2 x_2 \\
   \alpha (x_3)&=&2 x_3
  \end{array}
\end{array}
$$
with $[ x_2, x_1 ]$, $[x_3, x_1 ]$ and  $[
x_3,x_2 ]$ defined via skewsymmetry. It is not a
Lie algebra if and only if $a\neq0$ and $c\neq0$,
since
$[x_1,[x_2,x_3]]+[x_3,[x_1,x_2]]
+[x_2,[x_3,x_1]]= a c x_2.$
\end{example}
\begin{example}
[Twisted Heisenberg]\label{TwistedHeisenberg}
The usual Heisenberg  Lie algebra is twisted to  Hom-Lie algebras $(\mathfrak{g},[\cdot ,\cdot ],\alpha) $ defined, with respect to a basis  $\{x_1,x_2,x_3\}$, by the bracket 
$$[x_1,x_2]=(a_{11}a_{22}-a_{12}a_{21})x_3$$
and  linear maps  $\alpha $ are given with respect to the basis by the matrices
$$
\left(
\begin{array}{ccc}
a_{11} & a_{12} & 0 \\
a_{21}& a_{22}& 0 \\
a_{31} & a_{32}& a_{11}a_{22}-a_{12}a_{21}\\
\end{array}
\right),\;
$$  where $a_{11},a_{12},a_{21},a_{22} , a_{31},a_{32}\in \mathbb{K}.$
\end{example}

\begin{proposition}
\cite{Sheng} Given two Hom-Lie algebras $(\mathfrak{g},[\cdot ,\cdot ],\alpha) $
and $(\mathfrak{H},[\cdot ,\cdot],\beta)$, there is a Hom-Lie algebra
$(\mathfrak{g\oplus H},[\cdot ,\cdot ],\alpha+\beta),$ where the skew-symmetric bilinear
map
\[
\lbrack\cdot ,\cdot ]:\mathfrak{g\oplus H\times g\oplus H}\rightarrow\mathfrak{g\oplus
H}%
\]
is given by%
\[
\lbrack(x_{1},y_{1}),(x_{2},y_{2})]=([x_{1},x_{2}],[y_{1},y_{2}]),\forall
x_{1},x_{2}\mathfrak{\in g},y_{1},y_{2}\mathfrak{\in H},
\]
and the linear map
\[
\mathfrak{(\alpha+\beta):g\oplus H\rightarrow g\oplus H}%
\]
is given by
\[
\mathfrak{(\alpha+\beta)}(x,y)=(\mathfrak{\alpha}(x),\mathfrak{\beta
}(y)),\forall x\mathfrak{\in g},y\mathfrak{\in H}.
\]
\end{proposition}
Now, we give a characterization of Hom-Lie algebras morphism.
Let  $(\mathfrak{g,[\cdot,\cdot],\alpha})$ and
$(\mathfrak{H,[\cdot,\cdot ],\beta})$ be two Hom-Lie algebras and   $\phi:\mathfrak{g\rightarrow H} $ be a Hom-Lie algebras morphism.
Denote by $\mathfrak{g}_{\phi}\subset\mathfrak{g\oplus H}$ the graph of a
linear map $\phi:\mathfrak{g\rightarrow H}.$

\begin{proposition}[\cite{Sheng}] A linear map $\phi:\mathfrak{g\rightarrow H} $ is a Hom-Lie algebra morphism if and only if the graph $\mathfrak{g}_{\phi}\subset\mathfrak{g\oplus H}$ is a Hom-Lie subalgebra of $(\mathfrak{g\oplus H},[\cdot ,\cdot ],\alpha+\beta)$.
\end{proposition}

In the following we show that the center of a Hom-Lie algebra is an ideal for surjective maps.

\begin{definition}
Let $(\mathfrak{g},[\cdot ,\cdot],\alpha)$ be a multiplicative Hom-Lie algebra. Define the center of $\mathfrak{g}$, denoted by $\mathcal{Z}(\mathfrak{g})$, as
\[\mathcal{Z}(\mathfrak{g}) = \{z \in \mathfrak{g} : [x,z]=0, ~\forall x \in \mathfrak{g}\}.\]
\end{definition}

\begin{proposition}
Let $(\mathfrak{g},[\cdot ,  \cdot],\alpha)$ be a multiplicative Hom-Lie algebra. If $\alpha$ is surjective, then the center of $\mathfrak{g}$ is an ideal of $\mathfrak{g}$.
\end{proposition}
\begin{proof}
Let $z \in \mathcal{Z}(\mathfrak{g})$ and $x \in \mathfrak{g}$, we set $x = \alpha(u)$, then we have
\begin{align*}
[x,\alpha(z)]= [\alpha(u),\alpha(z)]
= \alpha([u,z])
=0,
\end{align*}
that is $\alpha(\mathcal{Z}(\mathfrak{g})) \subseteq \mathcal{Z}(\mathfrak{g})$.

Let $z \in \mathcal{Z}(\mathfrak{g})$ and $x,y\in \mathfrak{g}$,
 $
[x,[y,z]]= [x,0]
=0.
$ \\ 
Therefore, $\mathcal{Z}(\mathfrak{g})$ is an ideal of $\mathfrak{g}$.
\end{proof}

\subsection{Cohomology of Hom-Lie algebras}
A cohomology complex for Hom-Lie algebras was initiated in \cite{AM 2007} and then completed, in the multiplicative case, independently in \cite{AEM11,Sheng}.  We recall cohomology complex which is   relevant for deformation theory of Hom-Lie algebras \cite{AM 2007}.

Let $\left( \mathfrak{g},[\cdot,\cdot ],\alpha \right) $ be a multiplicative Hom-Lie algebra. A $\mathfrak{g}$-valued $p$-cochain is a $p$-linear alternating map $\varphi : \mathfrak{g}^p \rightarrow \mathfrak{g}$ satisfying $$\alpha \circ \varphi(x_0,...,x_{p-1})=\varphi\big(\alpha (x_0),\alpha(x_1),...,\alpha(x_{p-1})\big) \; \hbox{for all}\; x_0,x_1,...,x_{p-1} \in \mathfrak{g}.$$
The set of $\mathfrak{g}$-valued $p$-cochains is denoted by $\mathcal{C}_{HL}^p(\mathfrak{g},\mathfrak{g})$.\\
 
We define, 
 for $p\geq1$,  a $p$-coboundary operator of the Hom-Lie algebra $(\mathfrak{g},[\cdot ,\cdot ],\alpha)$ to be  the linear map $\delta_{HL}^{p} : \mathcal{C}_{HL}^p(\mathfrak{g},\mathfrak{g}) \rightarrow \mathcal{C}_{HL}^{p+1}(\mathfrak{g},\mathfrak{g})$ defined by
\begin{align}\label{HLcohomo}
\delta_{HL}^p\varphi(x_0,x_1,...,x_p)& =\sum_{k=0}^p(-1)^k \big[\alpha^{p-1}(x_k),\varphi(x_0,...,\widehat{x_k},...,x_p)\big]\\
 \ & +\sum_{0\leq i<j\leq n}\varphi([x_i,x_j],\alpha(x_0),...,\widehat{x_i},...,\widehat{x_j},...,\alpha(x_p)),\nonumber
\end{align}
where $\widehat{x_k}$ means  that $x_k$ is omitted.

The space of $n$-cocycles is defined by $$Z_{HL}^n(\mathfrak{g},\mathfrak{g}) = \ker \delta_{HL}^p=\{\varphi \in \mathcal{C}^n(\mathfrak{g},\mathfrak{g}):\ \delta_{HL}^n\varphi=0\},$$ and the space of $n$-coboundaries is defined by $$B_{HL}^n(\mathfrak{g},\mathfrak{g})= \operatorname{Im} \delta_{HL}^{p-1}=\{\psi=\delta_{HL}^{n-1}\varphi :\  \varphi \in \mathcal{C}^{n-1}(\mathfrak{g},\mathfrak{g}) \}.$$

One has $B_{HL}^n(\mathfrak{g},\mathfrak{g}) \subset Z_{HL}^n(\mathfrak{g},\mathfrak{g})$. The $n^{th}$ cohomology group of the Hom-Lie algebra $\mathfrak{g}$ is the quotient
$$H_{HL}^n(\mathfrak{g},\mathfrak{g})=\frac{Z_{HL}^n(\mathfrak{g},\mathfrak{g})}{B_{HL}^n(\mathfrak{g},\mathfrak{g})}.$$
The cohomology group is given by $H_{HL}(\mathfrak{g},\mathfrak{g})=\oplus_{p\geq 0} H_{HL}^p(\mathfrak{g},\mathfrak{g})$.

The set of cochains $\mathcal{C}(\mathfrak{g},\mathfrak{g})= \oplus_{p\geq 0}^{p}\mathcal{C}(\mathfrak{g},\mathfrak{g})$ may be endowed with a structure of graded Lie algebra using Nijenhuis-Richardson bracket \cite{AEM11}. In this paper we don't assume that the 2-cochains are multiplicative since we are dealing with the general case, see \cite{AM 2007}.

\section{ Solvable, Nilpotent and Filiform Hom-Lie algebras}
In this section, we introduce the definitions of nilpotent,  filiform and solvable Hom-Lie algebras. Then we discuss their properties.

\begin{definition}
Let $ (\mathfrak{g},[\cdot,\cdot],\alpha )$ be  an $n$-dimensional Hom-Lie
algebra. We set 
\begin{align}
& C^{0}\left(  \mathfrak{g}\right)     =\mathfrak{g}, \quad 
C^{m+1}\left(  \mathfrak{g}\right)     =\left[  \mathfrak{g}, C^{m}\left(\mathfrak{g}\right)  \right]
,\ \ Êm\geq 1.\label{nilpCond1}
\end{align}
We say that  $\ \left(  C^{i}\left(  \mathfrak{g}\right)  \right)  _{i\geq 0}$ constitutes a decreasing
central series  of \ $\mathfrak{g}$ if 
\begin{equation}
  \alpha (C^m(\mathfrak{g}))\subset C^m(\mathfrak{g}).\label{nilpCond2}
\end{equation}
\end{definition}

\begin{remark}
If the Hom-Lie algebra $(\mathfrak{g},[\cdot,\cdot],\alpha )$ is multiplicative then 
\eqref{nilpCond2} is useless. One may define the decreasing central series only using \eqref{nilpCond1}.
\end{remark}

\begin{definition}
\bigskip An  $n$-dimensional Hom-Lie algebra $\left(  \mathfrak{g},\left[  \cdot ,\cdot \right]  ,\alpha\right)$, with a decreasing central series,  is nilpotent if there exists an integer $m$ such that 
\[
C^{m}\left(  \mathfrak{g}\right)  =0\text{ and }C^{m-1}\left(  \mathfrak{g}\right)  \neq0.
\]
The integer $m$ is the nilindex. We say that this Hom-Lie algebra is filiform when 
$$\dim C^{k}(\mathfrak{g})=n-k-1\ \text{ with } 1\leq k\leq n-1.$$
\end{definition}

Notice that the nilindex is maximal when the Hom-Lie algebra is  filiform.
\begin{remark}
The nilindex of $n$-dimensional Hom-Lie algebra is smaller or equal to $n-1$.
\end{remark}

\begin{example} The twisted Heisenberg algebra of Example \ref{TwistedHeisenberg} is nilpotent when the 
 linear maps  $\alpha $ are  given with respect to the basis by the matrices
$$
\left(
\begin{array}{ccc}
a_{11} & 0 & 0 \\
a_{21}& a_{22}& 0 \\
a_{31} & a_{32}& a_{11}a_{22}\\
\end{array}
\right),\;
$$  where $a_{11},a_{21},a_{22} , a_{21},a_{22}\in \mathbb{K}.$
\end{example}

\begin{proposition}
Let $\left(  \mathfrak{g},\left[  \cdot ,\cdot \right]  ,\alpha \right)  $ be a multiplicative Hom-Lie algebra. Then $C^{1}\mathfrak{g}=\left[  \mathfrak{g},\mathfrak{g}\right]  $ and for $m\geq 1$, 
$C^{m+1}\mathfrak{g}=\left[  \mathfrak{g},C^{m}\mathfrak{g}\right]  \ $ are ideals.
\end{proposition}

\begin{proof}
Since $C^{1}\mathfrak{g}=\left[  \mathfrak{g},\mathfrak{g}\right]  $ $\subset \mathfrak{g}\ $and $\alpha(\left[  \mathfrak{g},\mathfrak{g}\right]
)=\left[ \alpha(\mathfrak{g}),\alpha(\mathfrak{g})\right]  \subset\left[  \mathfrak{g},\mathfrak{g}\right]  , $ thus $C^{1}\mathfrak{g}$
is an ideal. \newline 
Assume $C^{m}\mathfrak{g}$ is an ideal. We have $ \left[  \mathfrak{g},C^{m+1}\mathfrak{g}\right]=   \left[  \mathfrak{g},[\mathfrak{g},C^{m}\mathfrak{g}\right]\subset [\mathfrak{g},C^{m}\mathfrak{g}]= C^{m+1}\mathfrak{g},$
and $\alpha(C^{m+1}\mathfrak{g})=\alpha(\left[  \mathfrak{g},C^{m}\mathfrak{g}\right]
 )=\left[ \alpha(\mathfrak{g}),\alpha (C^{m}\mathfrak{g}  ) \right] \subset  \left[ \mathfrak{g},C^{m}\mathfrak{g}   \right]= C^{m+1}\mathfrak{g}.$
 Therefore $C^{m+1}\mathfrak{g}\ $is an ideal of $\mathfrak{g}.
$
\end{proof}
\begin{remark}
Notice  that when $\alpha~$is multiplicative, the central series $C^{n}(\mathfrak{g})$
constitutes a central series of ideals. 
\end{remark}

\begin{remark}
When the Hom-Lie algebra $\left(  \mathfrak{g},\left[  \cdot ,\cdot \right]  ,\alpha\right)$ is not necessarily multiplicative then the sequence $\left(  C^{m} \mathfrak{g} \right)  _{m}$ constitutes a descending central
series of ideals if $C^{m}  \mathfrak{g}$ are stable by $\alpha$ for $m>0$.
\end{remark}

\begin{proposition}
Let $\left(  \mathfrak{g},\left[  \cdot ,\cdot \right]  ,\alpha\right)$ be a Hom-Lie algebra, with a surjective structure map,  and 
 $A$, $B$  two ideals of $\mathfrak{g}$. The subspace $\left[
A,B\right]  $ of $\mathfrak{g}$, stable under $\alpha$, generated by the brackets $\left[  x,y\right]$, $x\in
A,y\in B$,  is still an ideal of $\mathfrak{g}$.
\end{proposition}
\begin{proof}Indeed, let $u\in \mathfrak{g}$ and $u=\alpha(u')$.
\[
\left[  u,\left[  A,B\right]  \right] =\left[  \alpha(u'),\left[  A,B\right]  \right] =\left[  \left[  u',A\right]  ,\alpha(B)\right]
+\left[  \alpha(A),\left[  u',B\right]  \right]\subset [A,B].
\]
\end{proof}

Now, we define the derived series of Hom-Lie algebras.
\begin{definition}
Let $(\mathfrak{g},[\cdot ,\cdot],\alpha)$ be a  multiplicative  Hom-Lie algebra, and $I$ an ideal of $\mathfrak{g}$. We define $D^r(I), r \in \mathbb{N}$, called the derived series of $I$,  by
\[D^0(I)=I \text{ and }D^{r+1}(I)=[D^r(I),D^r(I)].\]
\end{definition}
\begin{proposition}
The subspaces $D^r(I), r \in \mathbb{N}$, are subalgebras of $(\mathfrak{g},[\cdot , \cdot],\alpha)$.
\end{proposition}

\begin{proof}
We proceed by induction on $r \in \mathbb{N}$, the case of $r=0$ is trivial. Now suppose that $D^r(I)$ is a subalgebra of $\mathfrak{g}$. We prove that $D^{r+1}(I)$ is a subalgebra of $\mathfrak{g}$. Indeed
\begin{enumerate}
\item Let $y \in D^{r+1}(I)$, 
$
\alpha(y)=\alpha([y_{1},y_{2}])= [\alpha(y_{1}),\alpha(y_{2})] , \  y_{1},y_{2} \in D^r(I),
$ 
which is in $D^{r+1}(I)$ because $\alpha(y_{1}),\alpha(y_{2})\in D^r(I)$. That is $\alpha(D^{r+1}(I))\subseteq D^{r+1}(I)$.
\item  We have, for $x_1,x_2 \in D^{r+1}(I)$, 
$
[x_1,x_2] = [[x_{11},x_{12}],[x_{21},x_{22}]],$ where  $ x_{11},x_{12},x_{21},x_{22} \in D^r(I)$. Then  
$[x_1,x_2]\in D^{r+1}(I).$ \qedhere

\end{enumerate}
\end{proof}

\begin{proposition}
Let $(\mathfrak{g},[\cdot , \cdot],\alpha)$ be a multiplicative Hom-Lie algebra, and $I$ an ideal of $\mathfrak{g}$. If $\alpha$ is surjective, then $D^r(I),\, r \in \mathbb{N},$ are ideals of $\mathfrak{g}$.
\end{proposition}
\begin{proof}
We already have that $D^r(I), \, r \in \mathbb{N},$ are subalgebras.  We only need to prove that for all $x \in \mathfrak{g}$, and $y \in D^r(I)$, $[x,y]\in  D^r(I)$.

We proceed by induction over $r \in \mathbb{N}$, the case of $r=0$ is trivial. Now suppose that $D^r(I)$ is an ideal of $\mathfrak{g}$, we prove that $D^{r+1}(I)$ is an ideal of $\mathfrak{g}$:

Let $x \in \mathfrak{g}$ and $y\in D^{r+1}(I)$:
\begin{align*}
[x,y] &= [x,[y_1,y_2]] \quad y_1,y_2 \in D^r(I) \\
&= [\alpha(x'),[y_1,y_2]] \text{ for some } x\in \mathfrak{g} \\
&=  [[x',y_1],\alpha(y_2)]+[\alpha(y_1),[x',y_2]],\\
[x,y]& \in D^{r+1}(I)
\end{align*}
because all $\alpha(y_{i}) \in D^{r}(I)$ and all $[x,y_i]\in D^{r}(I)$ ($D^r(I)$ is an ideal), and then all the $[\alpha(y_i),[v,y_j]]$ are in  $D^{r+1}(I)$.
\end{proof}

\begin{definition}
Let $(\mathfrak{g},[\cdot , \cdot],\alpha)$ be a multiplicative Hom-Lie algebra, and $I$ an ideal of $\mathfrak{g}$. We define $C^r(I)$, $r \in \mathbb{N}$, the central descending series of $I$ by
\[C^0(I)=I \text{ and }C^{r+1}(I)=[I, C^r(I)].\]
\end{definition}

\begin{proposition}
Let $(\mathfrak{g},[\cdot,\cdot],\alpha)$ be a multiplicative Hom-Lie algebra, and $I$ an ideal of $\mathfrak{g}$. If $\alpha$ is surjective, then $C^r(I) ,\, r \in \mathbb{N},$ are ideals of $\mathfrak{g}$.
\end{proposition}
\begin{proof}
We proceed by induction over $r \in \mathbb{N}$, the case of $r=0$ is trivial. Now suppose that $C^r(I)$ is an ideal of $\mathfrak{g}$, we prove that $C^{r+1}(I)$ is an ideal of $\mathfrak{g}$:
\begin{enumerate}
\item Let $y \in C^{r+1}(I)$:
\begin{align*}
\alpha(y)&=\alpha([y_1,w])= [\alpha(y_{1}), \alpha(w)], \quad y_{1} \in I, ~~w \in C^r(I),
\end{align*}
which is in $C^{r+1}(I)$ because $\alpha(y_{1})\in I$ and $\alpha(w)\in C^r(I)$. That is, $\alpha(C^{r+1}(I))\subseteq C^{r+1}(I)$.
\item Let $x \in \mathfrak{g}$ and $y\in C^{r+1}(I)$:
\begin{align*}
[x,y] &= [x,[y_{1},w]] \quad y_1 \in I, w \in C^r(I) \\
&= [\alpha(v),[y_1,w]] \text{ for some } v \in \mathfrak{g} \\
&=   [\alpha(w),[y_1,v]]+ [\alpha(y_1),[v,w]],
 \end{align*}
which is in $C^{r+1}(I)$, because  $\alpha(y_1)\in I$, $\alpha(w) \in C^{r}(I)$, all $[v,y_1]\in I$ (I is an ideal) and $[v,w]\in C^{r}(I)$ ($C^r(I)$ is an ideal).
\end{enumerate}
\end{proof}

\begin{definition}
Let $(\mathfrak{g},[\cdot , \cdot],\alpha)$ be a multiplicative Hom-Lie algebra, and $I$ an ideal of $\mathfrak{g}$. The ideal $I$ is said to be solvable if there exists $r \in \mathbb{N}$ such that $D^r(I)=\{0\}$. It is said to be nilpotent if there exists $r \in \mathbb{N}$ such that $C^r(I)=\{0\}$.\\
In particular, the Hom-Lie algebra $\mathfrak{g}$ is solvable if it is solvable as an ideal of itself.
\end{definition}

\begin{remark}
Multiplicative nilpotent Hom-Lie algebras are solvable.
\end{remark}

\section{Constructions by twist of Hom-Lie algebras}
The following construction, introduced by Yau \cite{DY2009-1}, permits to construct a Hom-Lie algebra starting from a Lie algebra and a Lie algebra morphism. We discuss it in the case of nilpotent and filiform Lie algebras. Moreover we provide some examples.
\begin{theorem}[\cite{DY2009-1}]
Let $(\mathfrak{g},[\cdot ,\cdot ])$ be a Lie algebra and $\alpha:\mathfrak{g}\rightarrow
\mathfrak{g}\mathcal{\ }$a Lie algebra morphism, that is  $\alpha[x,y]=[\alpha(x),\alpha(y)].\ $Then
$(\mathfrak{g},[\cdot ,\cdot ]_{\alpha},\alpha)$ is a Hom-Lie algebra where  $[x,y]_{\alpha}=[\alpha
(x),\alpha(y)].$
\end{theorem}

Now, we consider the case of nilpotent algebras and filiform algebras. Moreover we use this twisting result to  construct nilpotent Hom-Lie algebras, which are of Lie type.

\begin{theorem}
Let $(\mathfrak{g},[\cdot ,\cdot ])$ be a nilpotent Lie algebra and $\alpha:\mathfrak{g}\rightarrow \mathfrak{g}\mathcal{\ }$a
Lie algebra morphism, that is  $\alpha\lbrack x,y]=[\alpha(x),\alpha(y)].\ $Then,
$(\mathfrak{g},[\cdot ,\cdot ]_{\alpha},\alpha)\ $\ is a nilpotent Hom-Lie algebra, where $\left[
x,y\right]  _{\alpha}=[\alpha(x),\alpha(y)]$.
\end{theorem}

\begin{proof}
One has $C_{\alpha}^{p}\mathfrak{g}=[\mathfrak{g},C_{\alpha
}^{p-1}\mathfrak{g}]_{\alpha}.$ We prove that $C_{\alpha}^{p}\mathfrak{g}\subset$ $C^{p}%
\mathfrak{g}$ by  induction.

The fact that $C_{\alpha}^{2}\mathfrak{g}=\alpha\lbrack \mathfrak{g},[\alpha(\mathfrak{g}),\alpha(\mathfrak{g})]]\subset\lbrack
\mathfrak{g},[\mathfrak{g},\mathfrak{g}]]=C^{2}\mathfrak{g}$ is a consequence of the fact that $\alpha\ $is a morphism of
$\mathfrak{g}$. \newline Indeed,
\[
\alpha\left(  C^{p}\left(  \mathfrak{g}\right)  \right)  =\alpha\left[ \mathfrak{g},C^{p-1}\left(
\mathfrak{g}\right)  \right]  =\left[  \alpha\left(  \mathfrak{g}\right)  ,\alpha\left(  C^{p-1}\left(
\mathfrak{g}\right)  \right)  \right]  \subset \left[  \mathfrak{g},\alpha\left(  C^{p-1}\left(  \mathfrak{g}\right)
\right)  \right]
\]%
\[
C_{\alpha}^{1}\left(  \mathfrak{g}\right)  =\alpha\left(  \left[  \mathfrak{g},\mathfrak{g}\right]  \right)
=\left[  \alpha\left(  \mathfrak{g}\right)  ,\alpha\left(  \mathfrak{g}\right)  \right]  \subset\left[
\mathfrak{g},\mathfrak{g}\right]  =C^{1}\left(  \mathfrak{g}\right).
\]
Suppose that $C_{\alpha}^{p-1}\left(  \mathfrak{g}\right)  \subset C^{p-1}\left(  \mathfrak{g}\right)
.$ We have,
\[
C_{\alpha}^{p}\left(  \mathfrak{g}\right)  =\left[  \mathfrak{g},C_{\alpha}^{p-1}\left(  \mathfrak{g}\right)
\right]  _{\alpha}\subset C^{p}\left(  \mathfrak{g}\right).
\]
It follows that, for all $p$,
\[
C_{\alpha}^{p}\left( \mathfrak{g}\right)  \subset C^{p}\left( \mathfrak{g}\right).
\]
\newline As $C^{n}\mathfrak{g}=0$, where $n\ $is the nilindex of ($\mathfrak{g},[\cdot ,\cdot ]$)\ and as
$C_{\alpha}^{p}\mathfrak{g}\subset C^{p}\mathfrak{g}\mathcal{\ }$for any $p$, we get $C_{\alpha}%
^{n}\mathfrak{g}=0.$ Thus $(\mathfrak{g},[\cdot ,\cdot ]_{\alpha},\alpha)\ $ is a nilpotent Hom-Lie
algebra.\newline
\end{proof}

However, the fact that $\ C_{\alpha}^{n-1}\mathfrak{g}\subset C^{n-1}\mathfrak{g}\neq0\ $\ implies
that it is possible that $C_{\alpha}^{n-1}\mathfrak{g}=0$. The nilindex of $\mathfrak{g}_\alpha$ is smaller or equal to the nilindex of $\mathfrak{g}$. Then the following question
arises: " Is the Hom-Lie algebra $(\mathfrak{g},[\cdot ,\cdot ]_{\alpha},\alpha)$ remains filiform or
only nilpotent?".  We have also the following result, which will be illustrated  by  two examples
below.

\begin{theorem}
Let $\left(  \mathfrak{g},\left[  \cdot ,\cdot \right]  \right)  $ be a filiform Lie algebra and
$\alpha$ a Lie algebra morphism. Then $\mathfrak{g}_{\alpha}=\left(  \mathfrak{g},\left[ \cdot ,\cdot \right]
_{\alpha},\alpha\right)  $ is a nilpotent  Hom-Lie algebra.
\\ If $\alpha$  is an automorphism then $\mathfrak{g}_{\alpha}$ is a filiform Hom-Lie algebra.
\end{theorem}

\begin{proof}
The proof is the same as the preceding one. It is enough  to prove that
$C_{\alpha}^{p}\left(  \mathfrak{g}\right)  \subset C^{p}\left(  \mathfrak{g}\right)  $ and it turns to be an equality  when  $\alpha$ is a bijection.
\end{proof}
More generally, there is a construction that leads to a new Hom-Lie algebra along with a Hom-Lie algebra and a weak Hom-Lie algebra morphism.
\begin{theorem}
Let $\left( \mathfrak{g},\left[  \cdot ,\cdot \right]  ,\alpha\right)  $ be a 
Hom-Lie algebra and $\beta$ a weak morphism of Hom-Lie algebra. Then ,
$\mathfrak{g}_\beta=\left(  \mathfrak{g},\beta\left[  \cdot ,\cdot \right]  _{\beta},\beta\alpha\right)  $ is a
Hom-Lie algebra.\\
If $\mathfrak{g}$  is nilpotent then $\mathfrak{g}_\beta$ is nilpotent. \\
If $\mathfrak{g}$ is filiform then $\mathfrak{g}_\beta$ is nilpotent, and it is  filiform when $\beta$ is an automorphism.
\end{theorem}


\begin{corollary}
Let $\left(  \mathfrak{g},\left[  \cdot ,\cdot \right]  ,\alpha\right)  $ be a multiplicative Hom-Lie
algebra and $n\geq0.$ The nth derived Hom-Lie algebra  defined by
$
\mathfrak{g}_{\left(  n\right)  }=\left(  \mathfrak{g},\left[  \cdot ,\cdot \right]  ^{(n)}=\alpha^{n}\circ \left[
\cdot ,\cdot \right]  ,\alpha^{n+1}\right)
$ is a Hom-Lie algebra.\\
 Moreover, if $\mathfrak{g}$
is filiform, then $\mathfrak{g}_{\left(  n\right)  }$ is nilpotent and it is filiform when $\alpha$ is a bijection.
\end{corollary}


\begin{theorem}
Let $\left(  \mathfrak{g},\left[  \cdot ,\cdot \right]  ,\alpha\right)  $ be a filiform Hom-Lie algebra
and $\alpha$ an algebra automorphism. Then, $\left(  \mathfrak{g},\left[ \cdot  ,\cdot \right]
_{\alpha^{-1}}\right)  $ is a filiform Lie algebra.\newline
\end{theorem}

\begin{example}
Consider the 4-dimensional filiform Lie algebra ($\mathfrak{g} ,[\cdot  ,\cdot])\ $defined by
\[
\left\{
\begin{array}
[c]{c}%
\left[  x_{1},x_{4}\right]  =x_{3}\\
\left[  x_{1},x_{3}\right]  =x_{2}%
\end{array}
\right.
\]
Let $\alpha$ be a linear map  such that, $\alpha(x_{i})=%
{\displaystyle\sum\limits_{k=1}^{4}}
\alpha _{ik}x_{k}, $ where $\alpha _{ik}$ are parameters. Among all the morphisms on $\mathfrak{g}$, we give hereafter two
examples.\newline For the morphism%
\[
\left\{
\begin{array}
[c]{c}%
 \alpha(x_{1})=\alpha_{12}x_{2}+\alpha_{13}x_{3}+\alpha_{14}x_{4}\\
 \alpha(x_{2})=0\\
 \alpha(x_{3})=0\\
  \alpha(x_{4})=\alpha_{42}x_{2}+\alpha_{43}x_{3}+\alpha_{44}x_{4}%
\end{array}
\right.
\]
The Hom-Lie algebra $\left(  \mathfrak{g}\ ,[\cdot ,\cdot ]_{\alpha},\alpha\right)  \ $\ is nilpotent but
not filiform.\newline Consider  the morphism,%
\[
\left\{
\begin{array}
[c]{c}%
\alpha(x_{1})=\alpha_{11}x_{1}+\alpha_{12}x_{2}+\alpha_{13}x_{3}+\alpha_{14}x_{4}\\
\alpha(x_{2})=\alpha_{11}\alpha_{33}x_{2}\\
\alpha(x_{3})=\alpha_{11}\alpha_{43}x_{2}+\alpha_{33}x_{3}\\
\alpha(x_{4})=\alpha_{42}x_{2}+\alpha_{43}x_{3}+\left(  \alpha_{33}/\alpha_{11}\right)
x_{4}%
\end{array}
\right.
\]
with $\alpha_{11}\alpha_{33}^{3}\neq0\ $ leading to  $\det(M_{\alpha})=\alpha_{11}\alpha_{33}%
^{3}\ \neq0.\ $\ Thus, $\alpha$ is an automorphism. It follows that,
($ \mathfrak{g} ,[\cdot ,\cdot ]_{\alpha},\alpha)\ $\ is a filiform Hom-Lie algebra.
\end{example}

\section{Filiform Hom-Lie algebras Varieties\bigskip}

Now, we adapt to Hom-type algebras the construction procedure initiated by Vergne \cite{vergne} and  described by Khakimdjanov in
\cite{Hakim1,Hakim2}, which was used to classify filiform Lie algebras up to dimension
$11$, see \cite{9}. Therefore, we provide a classification of  filiform Hom-Lie algebras up to dimension $7$.

\bigskip

Let $L_{n}$ be the $(n+1)$-dimensional Hom-Lie algebra defined by%
\begin{eqnarray}
& \lbrack x_{0},x_{i}]=x_{i+1},\  \  i=1,\cdots ,n-1\label{Ln}\\
& \text{and }\alpha=id,
\end{eqnarray}
where $\{x_{k}\}_{0\leq k\leq n}$ is a basis of $L_{n}$
(the undefined brackets are defined either by skew-symmetry or being zero). In the sequel, we refer to this bracket as $\mu_0$.\\  The Hom-Lie algebra $L_n$ is the model Filiform Hom-Lie algebra. It is  in a certain manner, the
simplest filiform Hom-Lie algebra. One may consider the structure map  $\alpha$ to be the identity map or any linear map associated to any lower trigonal matrix with respect to the basis $\{x_{k}\}_{0\leq k\leq n}$. They are given with respect to the basis, for $k=0, \cdots , n$,  by
$\alpha(x_k)=\sum_{i=k}^{n}\rho_{ik}x_i$, where $\rho_{i,k}$ are parameters in $\mathbb{K}$.\\

The Hom-Lie algebra $L_n$, with lower triangular structure map matrices is multiplicative  in the following cases:

\noindent \textbf{Case $L_2$:} The map $\alpha$ is defined, with respect to the basis $\{x_{k}\}_{0\leq k\leq n}$, as 
\begin{align*}
& \alpha (x_0)= \rho_{00}x_0+\rho_{10}x_1+\rho_{20}x_2,
\quad  \alpha (x_1)= \rho_{11}x_1+\rho_{21}x_2,\quad
 \alpha (x_2)= \rho_{00}\rho_{11} x_2.
\end{align*}
\noindent  \textbf{Case $L_n$ with $n\geq 3$ and $\rho_{00}\neq 0$ : }the map $\alpha$ is defined as 
\begin{align*}
& \alpha (x_0)= \sum_{i=0}^{n}\rho_{i,0}x_i,\\
& \alpha (x_1)=(\sum_{i=2}^{n-1}\frac{\rho_{i,2}}{\rho_{00}}x_{i-1})+ \rho_{n-1,1}x_{n-1}+\rho_{n,1}x_n,\\
& \alpha (x_2)=(\sum_{i=2}^{n-1}\rho_{i,2}x_{i})+ \rho_{00}\rho_{n-1,1}x_n,\\
& \alpha (x_k)=(\sum_{i=2}^{n-k+2}\rho_{00}^{k-2}\rho_{i,2}x_{k+i-2},\quad \text{  for }  k\geq 3.\\
\end{align*}
There is also another family of multiplicative Hom-Lie algebras $L_n$, corresponding to $\rho_{00}=0$,  which are given by  linear maps  $\alpha$ defined as
\begin{align*}
& \alpha (x_0)= \sum_{i=1}^{n}\rho_{i,0}x_i,\quad
 \alpha (x_1)=\sum_{i=1}^{n}\rho_{i,1}x_i,\quad
 \alpha (x_k)=0\ \text{  for }  k\geq 2.
\end{align*}

The approach developed by Vergne and Khakimdjanov to classify filiform Lie algebra and used in \cite{9} is based on the observation that any $(n+1)$-dimensional
filiform Lie algebra can be obtained as  a linear deformation of $L_{n}$ and it uses  the description of appropriate  cocycles that deform the bracket $\mu_0$ of the Lie algebra $L_n$. In the following, we discuss this construction for Hom-Lie algebras. \\
Let $\Delta$ be the set of pairs of integers $(k,r)$ such that $1\leq
k\leq n-1$ and  $2k+1<r\leq n$ (if $n$ is odd we suppose that $\Delta $  contains also the
pair $((n-1)/2,n)$). For any element $\ (k,r)\in\Delta$, we  associate a skew-symmetric 
2-cochain   denoted $\psi_{k,r}$ and defined by 
\begin{equation}\psi_{k,r}\left(  x_{i}%
,x_{j}\right)  =(-1)^{k-j}C_{j-k-1}^{k-i}x_{i+j+r-2k-1},~\label{PsiKR}
\end{equation} with $1\leq i\leq
k\leq j\leq n,\ $and $\psi_{k,r}\left(  x_{i},x_{j}\right)  =0$ elsewhere.  Moreover we
set $C_{0}^{0}=1$ and $C_{q}^{p}=0~$when $q<p.$


We aim to construct, up to isomorphism,  all $(n+1)$-dimensional filiform Hom-Lie algebras $(L_n,\mu, \alpha )$ given by a bracket $\mu$ and a structure map $\alpha$, where  
$\mu=\mu_{0}+\psi$ where  $\mu_{0}\ $ is the  bracket  in the model filiform Hom-Lie algebra $L_{n} $,  defined by $\mu_{0}\left(  x_{0}%
,x_{i}\right)  =x_{i+1},\  i=1,\cdots ,n-1,$ and $\psi\ $is 
defined by $\psi=%
{\displaystyle\sum\limits_{\left(  k,r \right)\in \Delta  }}
a_{k,r}\psi_{k,r}.$
The structure map $\alpha$ is given  by a  map, deforming the identity, corresponding, with respect to the basis, to a lower triangular  matrix of the form
\begin{small}
\[
\alpha=\left(
\begin{array}
[c]{ccccc}%
\rho_{00} & 0 &\cdots & 0 & 0\\
\rho_{10} & \rho_{11} & \cdots &  & 0\\%
\begin{array}
[c]{c}%
.\\
.
\end{array}
&  &  &  &
\begin{array}
[c]{c}%
.\\
.
\end{array}
\\
&  &  &  & 0\\
\rho_{n0} & \rho_{n1} &\cdots &  & \rho_{nn}%
\end{array}
\right).
\]
\end{small}
We define for two  skew-symmetric bilinear maps $\phi_1,\phi_2:\mathfrak{g}\times\mathfrak{g}\rightarrow \mathfrak{g} $ on a vector space $\mathfrak{g}$ with respect to a linear map $\alpha$,    the Gerstenhaber circle, which is a trilinear map defined as 
$$\phi_1\circ \phi_2 (x,y,z) =\phi_1(\phi_2 (x,y),\alpha (z))+\phi_1(\phi_2 (z,x),\alpha (y))+\phi_1(\phi_2 (y,z),\alpha (x)).$$
Therefore, an operation $\mu$ satisfies Hom-Jacobi condition is equivalent to $\mu\circ \mu=0$ and a 2-cochain $\psi$ is a 2-cocycle with respect to a Hom-Lie algebra $(\mathfrak{g},[\cdot, \cdot ],\alpha)$ if $[\cdot, \cdot ] \circ \psi+\psi\circ [,\cdot, \cdot ]=0.$\\
Using the Gerstenhaber circle, the bracket  $\ \mu=\mu_{0}+\psi$ satisfies the Hom-Jacobi identity if and only if
$\mu \circ \mu\left(  x_{i},x_{j},x_{k}\right)  =0\ \text{ for any } \ i,j,k \in %
\mathbb{N}
;$ that is 
\begin{eqnarray}\label{mu-psi}
&\left(  \mu_{0}+\psi\right)  \left(  \alpha\left(  x_{i}\right)  ,(\mu_{0}%
+\psi)\left(  x_{j},x_{k}\right)  \right)  +\left(  \mu_{0}+\psi\right)
\left(  \alpha\left(  x_{k}\right)  ,\mu_{0}+\psi\left(  x_{i},x_{j}\right)
\right)
\\
&+\left(  \mu_{0}+\psi\right)  \left(  \alpha\left(  x_{j}\right)  ,\mu_{0}%
+\psi\left(  x_{k},x_{i}\right)  \right)  =0.\nonumber
\end{eqnarray}
 Thus,  since $\mu_0$ satisfies Hom-Jacobi identity, it reduces using Gerstenhaber circle to  
 \begin{equation}\label{mu-psi2}
\delta^{2}\psi+\psi \circ \psi=0.
\end{equation} 


\begin{theorem}
Let
 $(L_n,\mu, \alpha )$ be an $(n+1)$-dimensional filiform Hom-Lie algebra given by a bracket 
$\mu=\mu_{0}+\psi$ where  $\mu_{0}$ is the  bracket  in the model filiform Hom-Lie algebra $L_{n} $,  defined by $\mu_{0}\left(  x_{0}%
,x_{i}\right)  =x_{i+1},\  i=1,\cdots ,n-1,$ the cochain  $\psi\ $is 
defined by $\psi=
{\displaystyle\sum\limits_{\left(  k,r \right)\in \Delta  }}
a_{k,r}\psi_{k,r}$ and 
the structure map $\alpha$ is given, with respect to the basis, by a lower triangular  matrix of the form
\begin{small}
\[
\alpha=\left(
\begin{array}
[c]{ccccc}%
\rho_{00} & 0 &\cdots & 0 & 0\\
0 & \rho_{11} & \cdots &  & 0\\%
\begin{array}
[c]{c}%
.\\
.
\end{array}
&  &  &  &
\begin{array}
[c]{c}%
.\\
.
\end{array}
\\
&  &  &  & 0\\
0 & \rho_{n1} &\cdots &  & \rho_{nn}%
\end{array}
\right).
\]
\end{small}
Then the bracket $ \mu=\mu_{0}+\psi$ satisfies the Hom-Jacobi identity if and only if
\begin{align}
&\delta^{2}\psi   ={\displaystyle\sum\limits_{\left(  k,r\right) \in\Delta   }}
a_{k,r}(\mu_{0}\circ\psi_{k,r}+\psi_{k,r}\circ\mu_{0})=0~\begin{small}(\psi \text{ is a 2-cocycle})\end{small},\\
&\psi\circ\psi   =0\ \begin{small}(\text{Hom-Jacobi condition})\end{small}.
\end{align}
\end{theorem}

\begin{proof}
Assume  that $\mu=\mu_{0}+\psi$ satisfies the Hom-Jacobi identity. Therefore we have the identity \eqref{mu-psi2}, which is evaluated for triples $(x_i,x_j,x_k)$. \\
If  $i,j,k\neq0,$ then $\delta^{2}\psi$ is always 0, which lead to $\psi\circ\psi   =0$.
If one of the indices is 0, then  $\psi\circ\psi   =0$, which leads to $\delta^{2}\psi=0.$
\end{proof}

Now, in the general situation, we consider one-parameter formal deformations introduced by Gerstenhaber for associative algebras \cite{M.G64} and by Nijenhuis-Richardson for Lie algebras \cite{NijenhuisRichardson}. This deformation theory  was extended to Hom-type algebras in \cite{AM 2007}. We assume that $\mu$ is an infinitesimal deformation, that is of the form $\mu=\mu_{0}+t \psi $, where $t$ is a parameter. Then $\mu$ satisfies Hom-Jacobi condition is equivalent to 
 \begin{equation}\label{mu-psi2bis}
\mu_0\circ\mu_0+t \delta^{2}\psi+t^2\psi \circ \psi=0.
\end{equation} 
Hence, by identification, it is equivalent to the system

\[
\left\{
\begin{array}
[c]{c}%
 \delta^{2}\psi   =0~\begin{small}(\psiÊ\text{ is  a 2-cocycle})\end{small},\\
 \psi\circ\psi   =0\ \begin{small}(\text{Hom-Jacobi condition})\end{small}.
\end{array}
\right.
\]

Now, we discuss these conditions on triples of basis elements $(x_i,x_j,x_k)$. \\

\noindent  \textbf{Case 1. $1\leq i< j<  k\leq n$}\\
Condition $\delta^{2}\psi=\mu_{0}\circ\psi+\psi\circ\mu_{0}=0$ is always true and $\psi\circ \psi =0$ is equivalent to 
\begin{equation}\label{psi0psi1}
\psi(\alpha(x_i),\psi(x_j,x_k))+\psi(\alpha(x_k),\psi(x_i,x_j))+\psi(\alpha(x_j),\psi(x_k,x_i))=0.
\end{equation}

\noindent \textbf{Case 2.  $1\leq i< j\leq n$ and $ k=0 $}\\
Condition  $\delta^{2}\psi=\mu_{0}\circ\psi+\psi\circ\mu_{0}=0$ is equivalent to 
\begin{equation}
\mu(\rho_{00} x_0,\psi(x_i,x_j))-\psi(\alpha(x_i),x_{j+1})+\psi(\alpha(x_j),x_{i+1})=0,\label{deltaPsi1}
\end{equation}
and condition  $\psi\circ \psi =0$ is equivalent to
\begin{equation}
\psi(\alpha (x_0),\psi(x_i,x_j))=0. \label{psiOpsi2}
\end{equation}

\begin{remark}
The maps  $\psi_{k,r}$, defined in \eqref{PsiKR}, are 2-cocycles if, for $i,j,k,r$ such that $1\leq
k\leq n-1$,   $2k+1<r\leq n$ and  $1\leq i\leq
k\leq j\leq n,\ $ we have 

\begin{eqnarray*}
& {\displaystyle\sum\limits_{e\geq j}}
\rho_{je}(-1)^{k-e}C_{i-k}^{k-e}x_{e+i+r-2k}+\rho_{00}(-1)^{k-i}%
C_{j-k-1}^{k-i}x_{j+i+r-2k}\\
 & -
{\displaystyle\sum\limits_{l\geq j}}
\rho_{lj}C_{j-k}^{k-l}x_{j+l+r-2k}=0.
\end{eqnarray*}
\end{remark}

\subsection{Filiform Hom-Lie algebras varieties and Adapted bases}
Let $\mathcal{F}_{m}$ be the variety of $m$-dimensional  filiform Hom-Lie algebras. Any $(n+1)$-dimensional filiform Hom-Lie
algebra bracket  $\mu\in \mathcal{F}_{n+1}$ is given by the bracket  $\mu_{0}+\psi$, where $\mu_{0}$
is the bracket  of the model filiform  Hom-Lie algebra $L_{n}$ defined in \eqref{Ln} and $\psi$ is a 2-cocycle defined by $\psi=
{\displaystyle\sum\limits_{\left(  k,r\right)  }}
a_{k,r}\psi_{k,r},$ where $\psi_{k,r}$ are defined in \eqref{PsiKR}, and satisfying  $\psi \circ \psi=0$.

We aim, for a given dimension,  to find the 2-cochains $\psi$ such that the conditions  $\delta^{2}\psi=0$ and $\psi \circ \psi=0$ are satisfied. In order to simplify  the calculation, we use the concept of adapted basis. We say that a basis is  \emph{adapted} if 
the bracket can be written as  $\mu=\mu_{0}+\psi$ in this basis and the matrix associated to the structure map $\alpha$  is still lower triangular. 

In the sequel, following \cite{9} we study adapted basis changes.

\begin{proposition}
Let $(\mathfrak{g},[\cdot ,\cdot ],\alpha )$ be an $(n+1)$-dimensional filiform Hom-Lie algebra and $\{ x_{i}\}_{i=0,\cdots n}$ be an adapted basis. 
Let $f\in End\left(  \mathfrak{g}\right)  $ be an invertible map  such
that $\{  f\left(  x_{0}\right)  ,f\left(  x_{1}\right)  ,\cdots,f\left(
x_{n}\right)  \}  $ is an adapted basis. Then,
\begin{align*}
f\left(  x_{0}\right)   &  =a_{0}x_{0}+a_{1}x_{1}+\cdots +a_{n}x_{n}\\
f\left(  x_{1}\right)   &  =b_{0}x_{0}+b_{1}x_{1}+\cdots +b_{n}x_{n}\\
&  ...\\
f\left(  x_{i}\right)   &  =\left[  f\left(  x_{0}\right)  ,f\left(
x_{i-1}\right)  \right],
\end{align*}
where $a_{0}b_{1}\neq0.$
\end{proposition}
\begin{proof}
We have for $i\geq2$, $a_{0}\neq0$ and \ $\left(  a_{0}b_{1}-a_{1}b_{0}\right)  \neq
0$,
\[
f\left(  x_{i}\right)  =\left[  f\left(  x_{0}\right)  ,f\left(
x_{i-1}\right)  \right]  =a_{0}^{i-2}\left(  a_{0}b_{1}-a_{1}b_{0}\right)
x_{i}+%
{\displaystyle\sum\limits_{j\geq i+1}}
I_{j}x_{j},%
\]
since the Hom-Lie algebra  is filiform. It follows that,
\begin{equation}
\left[  f\left(  x_{1}\right)  ,f\left(  x_{2}\right)  \right]  =b_{0}\left(
a_{0}b_{1}-a_{1}b_{0}\right)  x_{3}+%
{\displaystyle\sum\limits_{j\geq 4}}
\lambda_{j}x_{j} .\label{ad}%
\end{equation}
When $b_{0}\neq0,$ we have
\[
\left[  f\left(  x_{1}\right)  ,f\left(  x_{2}\right)  \right]  =\lambda
f\left(  x_{3}\right)  +%
{\displaystyle\sum\limits_{j\geq4}}
e_{j}f\left(  x_{j}\right),
\]
with $\lambda\neq0.$ Thus, $\left(  f\left(  x_{i}\right)  \right)  _{i=0,\cdots , n}$
is not adapted. This proves that the relation $\left(  \ref{ad}\right)  $ is possible  only when $b_{0}=0.$\newline Conversely, an arbitrary choice of
scalars $a_{0},a_{1},\cdots,a_{n},b_{1},\cdots , b_{n}$ with $a_{0}b_{1}\neq 0$,
determines an invertible map  $f\in End\left(  \mathfrak{g}\right)  $ such that the basis $$\{
f\left(  x_{0}\right)  ,f\left(  x_{1}\right)  ,\cdots ,f\left(  x_{n}
\right) \} $$ is adapted.

The fact that the structure map is given by a lower triangular matrix is straightforward.
\end{proof}

\begin{remark}
 A basis change $f\in End\left(  \mathfrak{g}\right) 
$ is said to be adapted, with respect to  the bracket  $\mu$ and the structure map $\alpha$  if the image of an adapted basis is
an adapted basis.

The set of adapted basis changes in $\mathcal{F}_{m}$ is a closed subgroup of the linear group $GL_m(\mathfrak{g})$ and  it is denoted by $GL_{ad}(\mathfrak{g})$. Moreover,  $dim\ GL_{ad}(\mathfrak{g})=2m+1$.
\end{remark}

\subsection{ Elementary Basis changes}
The following types of basis changes are said to be
elementary:\newline\textbf{ Type 1} : Let $2\leq k\leq n,\ b\in%
\mathbb{K}
$. We define 
$
\sigma\left(  b,k\right)$  by $$x_{0}^{\prime}=x_{0},\ Êx_{1}%
^{\prime}=x_{1}+bx_{k},\ x_{i}^{\prime}=\left[  x_{0}^{\prime}%
,x_{i-1}^{\prime}\right] \text{ for } 2\leq i\leq n.
$$
\newline \textbf{Type 2}: Let $1\leq k\leq n,\ a\in%
\mathbb{K}
$. We define 
$
\tau\left(  a,k\right) $ by $$x_{0}^{\prime
}=x_{0}+ax_{k},\ x_{1}^{\prime}=x_{1},\  x_{i}^{\prime}=\left[  x_{0}^{\prime},x_{i-1}^{\prime
}\right]  \text{ for } 2\leq i\leq n.
$$
\newline \textbf{Type 3}: Let $a,b\in%
\mathbb{K}
^{\ast}$. We define 
$
\upsilon\left(  a,b\right)$ by $$x_{0}^{\prime}=ax_{0},\ x_{1}^{\prime}=b x_{1},\ x_{i}^{\prime}=\left[  x_{0}^{\prime}%
,x_{i-1}^{\prime}\right]  \text{ for } 2\leq i\leq n.
$$

The following theorem shows that any adapted basis change  can be
represented as a composition of  elementary basis changes. 

\begin{theorem}
[\cite{9}] Let $f$ be an element of $GL_{ad}(\mathfrak{g})$. Then $f$ is a composition  of
elementary basis changes; that is, if $f$ is given by
\begin{align*}
f\left(  x_{0}\right)   &  =a_{0}x_{0}+a_{1}x_{1}+...a_{n}x_{n}\\
f\left(  x_{1}\right)   &  =b_{0}x_{0}+b_{1}x_{1}+...b_{n}x_{n}\\
&  ...\\
f\left(  x_{i}\right)   &  =\left[  f\left(  x_{0}\right)  ,f\left(
x_{i-1}\right)  \right]
\end{align*}
Then,  $f$ can be represented as a composition 
\begin{align*}
\allowbreak f  &  =\upsilon\left(  a_{0},b_{1}\right)  \circ\tau\left(
a_{n}/a_{0},n\right)  \circ\tau\left(  a_{n-1}/a_{0},n-1\right)  \circ
\cdots \circ\ \tau\left(  a_{1}/a_{0},1\right) \\
&  \circ\sigma\left(  b_{n}/b_{1},n\right)  \circ\sigma\left(  b_{n-1}%
/b_{0},n-1\right)  \circ\cdots \circ\sigma\left(  b_{2}/b_{0},2\right).
\end{align*}
 
\end{theorem}

\begin{proof}
Direct calculations, see \cite{9}.
\end{proof}
\begin{remark}
The decomposition of   an adapted basis change is not unique. For example, some products commute.
 We may  have
\begin{align*}
f  &  =\tau\left(  a_{n},n\right)  \circ\tau\left(  a_{n-1},n-1\right)
\circ\cdots \circ\tau\left(  a_{1},1\right) \\
&  \circ\sigma\left(  b_{n},n\right)  \circ\sigma\left(  b_{n-1},n-1\right)
\circ\cdots \circ\sigma\left(  b_{2},2\right)  \allowbreak\circ\upsilon\left(
a_{0},b_{1}\right)
\end{align*}
or
\begin{align*}
f  &  =\upsilon\left(  a_{0},b_{1}\right)  \circ\tau\left(  a_{n}%
/a_{0},n\right)  \circ\tau\left(  a_{n-1}/a_{0},n-1\right)  \circ
\cdots \circ\ \tau\left(  a_{2}/a_{0},2\right) \\
&  \circ\sigma\left(  b_{n}/b_{1},n\right)  \circ\sigma\left(  b_{n-1}%
/b_{0},n-1\right)  \circ\cdots \circ\ \sigma\left(  b_{2}/b_{0},2\right)
\circ\ \tau\left(  a_{1}/a_{0},1\right)
\end{align*}
or
\begin{align*}
f  &  =\tau\left(  a_{n},n\right)  \circ\tau\left(  a_{n-1},n-1\right)
\circ\cdots \circ\tau\left(  a_{2},2\right) \\
&  \circ\sigma\left(  b_{n},n\right)  \circ\sigma\left(  b_{n-1},n-1\right)
\circ\cdots \circ\sigma\left(  b_{2},2\right)  \allowbreak\circ\tau\left(
a_{1},1\right)  \circ\upsilon\left(  a_{0},b_{1}\right).
\end{align*}

\end{remark}
\subsection{ Classification Method for $m$-dimensional  Filiform Hom-Lie algebras }
Here, we describe  classification's method for  $m$-dimensional filiform
Hom-Lie algebra based on  adapted bases. 
In order to provide the classification of  filiform Hom-Lie algebras of dimension $m$ $\leq7$, we  apply the following algorithm. 
\begin{enumerate}
\item Fix the form of $\psi$ and solve conditions \eqref{psi0psi1},\eqref{deltaPsi1},\eqref{psiOpsi2}.
\item  Consider the regular case,
where the first free elements (non-depending on the others elements) in $\psi$ are
non-zero.
\item Using the elementary basis change $v(a,b$),  transform
two (one in the homogeneous case) of non-zero coefficients to 1.
\item  Apply successively  elementary basis changes and eliminate the most
possible parameters. The elimination of parameters should keep 
that eliminated parameters equal  to  zero in the next basis changes.
\item Use all possible basis changes which conserve the
parameters transformed into 1 or 0.
\item  Write the new matrix of the
linear map associated to the Hom-Lie algebra after every basis change. We
then obtain the final structure map matrix of the simplified Hom-Lie algebra.
\item Steps
(3) and  (4) can be permuted.
\end{enumerate}

In the sequel, we apply these steps, first to  $m$-dimensional filiform Hom-Lie algebras with $m\leq 5$ and then successively  to dimension $6$ and $7$.

\section{Filiform Hom-Lie algebras of dimension $\leq5$}

Any filiform Hom-Lie algebra bracket  of dimension less or equal to $5$ is
isomorphic to $\mu_{0}+\psi,\ $where $\mu_{0}$ is the bracket  of $L_{n} 
$, defined in \ref{Ln},  $(n=2,3,4)\ $and,%
\begin{align*}
\psi &  =0\ \ \text{if }n=2\text{ or }n=3\\
\psi &  =a_{14}\psi_{14}\text{ if }n=4.
\end{align*}
Conversely, for an arbitrary choice of parameters $\{a_{k,r}\},$ $\mu_{0}+\psi$
is a Hom-Lie algebra bracket.

\begin{theorem}[Classification]
 Every $n$-dimensional filiform Hom-Lie algebra with $n\leq 5$  is
isomorphic to one of the following pairwise non isomorphic Hom-Lie algebras $(\mathfrak{g},[\cdot,\cdot], \alpha)$, where the bracket is defined by the bracket $\mu_{m}=\mu_{0}+\psi $ and the linear map $\alpha$ is given,  with respect to the same basis,  by any lower  triangular  matrix with
arbitrary elements.
 
\noindent \emph{Dimension 3}: 
$\blacktriangleright$ $\mu_{3}^{1}:\mu_{0}$. 
\newline\emph{Dimension 4}:
$\blacktriangleright$ $\mu_{4}^{1}:\ \mu_{0} $.
 \newline\emph{Dimension 5}:

\begin{itemize}
\item[$\blacktriangleright$] $\mu_{5}^{1}:\mu_{0}\ $,  
\item[$\blacktriangleright$] $\mu_{5}^{2}:\mu_{0}+\psi_{1,4}$.
\end{itemize}
\end{theorem}
\begin{proof}Straightforward.
\end{proof}

\section{Filiform Hom-Lie algebras of dimension $6$}

Any filiform Hom-Lie algebra bracket of dimension $6$ is isomorphic to
$\mu_{0}+\psi,\ $where $\mu_{0}$ is the bracket of $L_{5}$ and,%
\[
\ \psi=a_{1,4}\psi_{1,4}+a_{1,5}\psi_{1,5}+a_{2,5}\psi_{2,5}.
\]

We give conditions that make the bracket $\mu=\mu_{0}+\psi$ and the structure
map $\alpha$ \ characterizing a filiform Hom-Lie algebra.

Condition $\delta^{2}\psi=0$   is equivalent
to $\rho_{11}C_{1}^{1}=0;\ \rho_{23}+\rho_{12}=0;\ \rho_{00}%
=\rho_{11}C_{1}^{0}$ and $\psi\circ\psi=0$ is equivalent to $\rho
_{01}a_{14}a_{25}C_{1}^{1}=0$. Therefore, we have the following cases

\begin{enumerate}
\item Case $C_{1}^{1}=0$. It leads to $\rho_{23}=-\rho_{12};\ \rho
_{00}=\rho_{11}C_{1}^{0}$.

\item Case $C_{1}^{1}\neq0$. Then $\rho_{11}=0$ and $\rho_{00}=0$.
Moreover we have the following three subcases

\begin{enumerate}
\item Case $\rho_{01}=0$,

\item Case $a_{14}=0$,

\item Case $a_{25}=0$.
\end{enumerate}
\end{enumerate}


\subsection{Basis Changes and elimination of parameters}

The use of the basis change $\sigma\left(  b,2\right)  $ gives a new bracket with $\psi=b_{14}\psi_{14}+b_{15}\psi_{15}+b_{25}\psi
_{25}$. The new coefficients 
$B=\{b_{k,m}\}$ satisfy
\begin{align*}
b_{14} &  =a_{14},\ b_{25}=a_{25},\\
b_{15} &  =a_{15}+b^{2}a_{25}+b\left(  C_{1}^{0}-1\right)  a_{14}=:P_{1}(b).
\end{align*}
The use of the basis change $\sigma\left(  b,3\right)  $ gives a new bracket with coefficients 
$B=$ $\{b_{k,m}\}$ such that
\begin{align*}
b_{14} &  =a_{14},\ b_{25}=a_{25},\\
b_{15} &  =a_{15}-ba_{25}\left(  1+C_{1}^{1}\right)  =:P_{2}(b).
\end{align*}
The use of the basis change $\tau\left(  a,1\right)  $ gives a new bracket with coefficients 
$B=$ $\{b_{k,m}\}$ such that
\begin{align*}
b_{14} &  =a_{14},\\
b_{15} &  =a_{15}-aa_{14}^{2}\left(  1+C_{1}^{0}\right)  +a^{2}a_{14}%
^{2}a_{25}C_{1}^{1}=:P_{3}(a),\\
b_{25} &  =\frac{a_{25}}{1-aa_{25}C_{1}^{1}}.%
\end{align*}
The use of the basis change $\tau\left(  a,2\right)  $ gives a new bracket with coefficients 
$B=$ $\{b_{k,m}\}$ such that
\begin{align*}
b_{14} &  =a_{14},\ b_{25}=a_{25},\\
b_{15} &  =a_{15}+a\ a_{14}a_{25}C_{1}^{1}-a\ a_{14}a_{25}=:P_{4}(a).
\end{align*}
The use of the basis change $\nu\left(  a,b\right)  $ gives a new bracket with coefficients 
$B=$ $\{b_{k,m}\}\ $ such that
\[
b_{k,m}=b\frac{a_{k,m}}{a^{m-2k}}.%
\]

\begin{remark}
In the previous basis changes, one may reduce $a_{15}$ to zero if the
corresponding polynomials $P_1,\cdots , P_4$ admit a root.
\end{remark}

\subsection{Structure map and basis changes}

We illustrate the impact of  basis changes on the matrix of the structure map
$\alpha$. Consider the bracket  $\mu=\mu_{0}+a_{14}\psi_{14}+a_{15}\psi_{15}+a_{25}\psi
_{25} ,$ with $ C_{1}^{1}=0$  and the structure  map $\rho$ such that $\rho(x_{i})=\sum
\limits_{j\geq i}\rho_{ij}x_{j}$ and satisfying  $$\rho_{11}C_{1}^{0}=\rho_{00},\ \rho
_{11}C_{1}^{1}=0\ ,\rho_{23}+\rho_{12}=0.$$

Using $\nu(a_{1},b_{1})$ and then $\sigma(b,3)$ with $b=a_{15}a_{25}%
/a_{14}^{2}$ and $a_{1}=\frac{a_{15}}{a_{25}},\ b_{1}=\frac{a_{14}}{a_{25}%
^{2}}$, we obtain that $\mu\cong\mu_{0}+\psi_{14}+\psi_{25}\ $ and the structure
map $\rho$ is
 defined with respect to the new basis as
\begin{align*}
&
\begin{array}
[c]{cc}%
\rho\left(  x_{0}^{{\prime}}\right)  = & \rho_{00}x_{0}^{\prime
}+\dfrac{a_{1}}{b_{1}}\rho_{01}x_{1}^{{\prime}}+\dfrac{\rho_{02}}{b_{1}}%
x_{2}^{{\prime}}+\dfrac{\rho_{03}}{a_{1}b_{1}}x_{3}^{{\prime}}+\dfrac
{\rho_{04}}{a_{1}^{2}b_{1}}x_{4}^{{\prime}}+\dfrac{\rho_{05}}{a_{1}^{3}b_{1}%
}x_{5}^{{\prime}}%
\end{array}
\\
&
\begin{array}
[c]{cc}%
\rho\left(  x_{1}^{{\prime}}\right)  = & \rho_{11}x_{1}^{{\prime}%
}+\dfrac{\rho_{12}}{a_{1}}x_{2}^{{\prime}}+\dfrac{\rho_{13}}{a_{1}^{2}}%
x_{3}^{{\prime}}+\dfrac{\rho_{14}}{a_{1}^{3}b_{1}}x_{4}^{{\prime}}%
+\dfrac{\rho_{15}}{a_{1}^{4}}x_{5}^{{\prime}}%
\end{array}
\\
&
\begin{array}
[c]{cc}%
\rho\left(  x_{2}^{{\prime}}\right)  = & \rho_{22}x_{2}^{{\prime}%
}+\dfrac{\rho_{23}}{a_{1}}x_{3}^{{\prime}}+\dfrac{\rho_{24}}{a_{1}^{2}}%
x_{4}^{{\prime}}+\dfrac{\rho_{25}}{a_{1}^{3}}x_{5}^{{\prime}}%
\end{array}
\\
&
\begin{array}
[c]{cc}%
\rho\left(  x_{3}^{{\prime}}\right)  = & \rho_{33}x_{3}^{{\prime}%
}+\dfrac{\rho_{34}}{a_{1}}x_{4}+\dfrac{\rho_{35}}{a_{1}^{2}}x_{5}^{{\prime}}%
\end{array}
\\
&
\begin{array}
[c]{cc}%
\rho\left(  x_{4}^{{\prime}}\right)  = & \rho_{44}x_{4}^{{\prime}%
}+\dfrac{\rho_{45}}{a_{1}}x_{5}^{{\prime}}%
\end{array}
\\
&
\begin{array}
[c]{cc}%
\rho\left(  x_{5}^{\prime}\right)  = & \rho_{55}x_{5}^{\prime}.%
\end{array}
\end{align*}
We denote the entries of the matrix  of $\rho$ with respect to the basis $\{x_{i}^{\prime}\}_{i=0,\cdots , 5}$ by $\rho_{ij}^{\prime}$. 
Therefore, we use $\sigma(b,3)$ and finally the structure map $\rho$  is given with respect to the new basis $\{x_{i}^{\prime\prime}\}_{i=0,\cdots , 5}$  by 
\begin{align*}
\rho\left(  x_{0}^{\prime\prime}\right)   &  =\rho_{00}^{\prime}x_{0}^{\prime\prime}+\rho
_{01}^{\prime}x_{1}^{\prime\prime}+\rho_{02}^{\prime}x_{2}^{\prime\prime}+\left(  \rho_{03}^{\prime
}-b\rho_{01}^{\prime}\right)  x_{3}^{\prime\prime}+\left(  \rho_{04}^{\prime}-b\rho
_{02}^{\prime}\right)  x_{4}^{\prime\prime}\\
&  +\left(  \rho_{05}^{\prime}-b\rho_{03}^{\prime}+b^{2}\rho_{01}^{\prime
}\right)  x_{5}^{\prime\prime}\\
\rho\left(  x_{1}^{\prime\prime}\right)   &  =\rho_{11}^{\prime}x_{1}%
^{\prime\prime}+\left(  \rho_{22}^{\prime}b+\rho_{12}^{\prime}\right)  x_{2}^{\prime\prime}+\left(
\rho_{13}^{\prime}-b\rho_{11}^{\prime}+b\rho_{23}^{\prime}\right)  x_{3}^{\prime\prime}\\
&  +\left(  \rho_{14}^{\prime}-b\rho_{12}^{\prime}-b^{2}\rho_{22}^{\prime
}+b\rho_{24}^{\prime}\right)  x_{4}^{\prime\prime}\\
&  +\left(  \rho_{15}^{\prime}-b\rho_{13}^{\prime}+b^{2}\rho_{11}^{\prime
}-b^{2}\rho_{23}^{\prime}+b^{2}\rho_{25}^{\prime}\right)  x_{5}^{\prime\prime}\\
\rho^{\prime\prime}\left(  x_{2}^{\prime\prime}\right)   &  =\rho_{22}^{\prime}x_{2}%
^{\prime\prime}+\left(  \rho_{33}^{\prime}b+\rho_{22}^{\prime}\right)  x_{3}^{\prime\prime}+\left(
\rho_{24}^{\prime}-b\rho_{22}^{\prime}+b\rho_{34}^{\prime}\right)  x_{4}^{\prime\prime}\\
&  +\left(  \rho_{25}^{\prime}-b\rho_{23}^{\prime}-b^{2}\rho_{33}^{\prime
}+b\rho_{35}^{\prime}\right)  x_{5}^{\prime\prime}\\
\rho\left(  x_{3}^{\prime\prime}\right)   &  =\rho_{33}^{\prime}x_{3}%
^{\prime\prime}+\left(  \rho_{44}^{\prime}b+\rho_{34}^{\prime}\right)  x_{4}^{\prime\prime}+\left(
\rho_{35}^{\prime}-b\rho_{33}^{\prime}+b\rho_{45}^{\prime}\right)  x_{5}^{\prime\prime}\\
\rho\left(  x_{4}^{\prime\prime}\right)   &  =\rho_{44}^{\prime}x_{4}%
^{"}+\rho_{45}^{\prime}x_{5}^{"}\\
\rho\left(  x_{5}^{\prime\prime}\right)   &  =\rho_{55}^{\prime}x_{5}^{\prime\prime}.%
\end{align*}

\subsection{Classification}

\begin{theorem}[6-dimensional Classification]
Every $6$-dimensional filiform Hom-Lie algebra is isomorphic to one of the
following pairwise non isomorphic Hom-Lie algebras $(\mathfrak{g},[\cdot
,\cdot],\alpha)$, where the bracket is defined by the multiplication $\mu
_{6}^{i}=\mu_{0}+\psi_{6}^{i}\ $and the linear map $\alpha$ is given by its
corresponding matrix with respect to the same basis:

%

\begin{itemize}

\item[$\blacktriangleright$] $\mu_{6}^{1}:\ \mu_{0},\ $with a structure map $\rho$ given by any lower
triangular matrix.

\item[$\blacktriangleright$] $\mu_{6}^{2}:\ \mu_{0}+\psi_{14}+\psi_{25},$ with a structure map $\rho$ given
by  matrices of the form
\[
\left(
\begin{array}
[c]{cccccc}%
\rho_{00} & 0 & 0 & 0 & 0 & 0\\
\rho_{01} & \rho_{11} & 0 & 0 & 0 & 0\\
\rho_{02} & \rho_{12} & \rho_{22} & 0 & 0 & 0\\
\rho_{03} & \rho_{13} & \rho_{00}-C_{1}^{1}\rho_{01}-C_{1}^{0}\rho_{11}%
-\rho_{12} & C_{1}^{1}\rho_{11} & 0 & 0\\
\rho_{04} & \rho_{14} & \rho_{24} & \rho_{34} & \rho_{44} & 0\\
\rho_{05} & \rho_{15} & \rho_{25} & \rho_{35} & \rho_{45} & \rho_{55}%
\end{array}
\right).
\]
\item[$\blacktriangleright$] $\mu_{6}^{3}:\ \mu_{0}+\psi_{14}$ , with $\ C_{1}^{0}\neq0,$ and the structure
map $\rho$ given by matrices of the form
\[
\left(
\begin{array}
[c]{cccccc}%
\rho_{00} & 0 & 0 & 0 & 0 & 0\\
\rho_{01} & \rho_{00}/C_{1}^{0} & 0 & 0 & 0 & 0\\
\rho_{02} & \rho_{12} & \rho_{22} & 0 & 0 & 0\\
\rho_{03} & \rho_{13} & \rho_{23} & \rho_{33} & 0 & 0\\
\rho_{04} & \rho_{14} & \rho_{24} & \rho_{34} & \rho_{44} & 0\\
\rho_{05} & \rho_{15} & \rho_{25} & \rho_{35} & \rho_{45} & \rho_{55}%
\end{array}
\right).
\]
\item[$\blacktriangleright$] $\mu_{6}^{4}:\ \mu_{0}+\psi_{15},$ with the structure map $\rho$ given by any
lower triangular matrix.
\end{itemize}

\end{theorem}

\begin{proof}
[ Sketch of proof] We study the following cases.\\ $\bullet$ Case $a_{14}\neq0,\ a_{25}\neq0 $ and$~C_{1}^{1}%
\neq\{-1,0\},$ $\rho_{01}=0.\ $We use $\sigma\left(  b,3\right)  $ and then
$\nu\left(  a,b^{\prime}\right)$. So, we obtain%
\[
\psi_{6}^{1}=\psi_{14}+\psi_{25}.
\]
$\bullet$ Case  $~a_{14}\neq0,\ a_{25}\neq0,\ C_{1}^{1}=0.\ $We use $\sigma\left(
b,3\right)  $ and then $\nu\left(  a,b^{\prime}\right) $. So, we obtain%
\[
\psi_{6}^{2}=\psi_{14}+\psi_{25}.
\]
$\bullet$ Case $~a_{14}\neq0,\ a_{25}\neq0,\ C_{1}^{1}=-1,$ $C_{2}^{0}%
a_{14}=a_{26}C_{1}^{1}.\ $We use $\tau\left(  a,2\right)  $, then $\nu\left(
a^{\prime},b\right)  $. So,  we obtain%
\[
\psi_{6}^{3}=\psi_{14}+\psi_{25}.
\]
$\bullet~$ Case  $a_{14}\neq0,a_{25}=0\ $and$~C_{1}^{0}\neq1$. We use
$\sigma\left(  b,2\right)  $ and then $\nu\left(  a,b^{\prime}\right)  $. So, 
we obtain%
\[
\psi_{6}^{4}=\psi_{14}.
\]
$\bullet~$ Case  $a_{14}\neq0,a_{25}=0\ $and$~C_{1}^{0}=1$.\ We use $\tau\left(
a,2\right)  $ and then $\nu\left(  a^{\prime},b\right)  $. So,  we obtain%
\[
\psi_{6}^{5}=\psi_{14}.
\]
$\bullet$ Case  $a_{14}=0,a_{25}\neq0,a_{15}=0,a_{16}\neq0.$ We use
$\sigma\left(  b,2\right)  $ and $\nu\left(  a,b^{\prime}\right)  $. So,  we
obtain%
\[
\psi_{6}^{6}=\psi_{25}.
\]
$\bullet$ Case  $a_{14}=0 , a_{25}=0, a_{15}\neq0.\ $We use $\nu\left(
a,b\right) $. So,  we obtain%
\[
\psi_{6}^{7}=\psi_{15}.
\]
$\bullet$ Case  $a_{14}=0,\ a_{26}\neq0,a_{15}\neq0,\ C_{1}^{1}=-1.\ $We
use $\tau\left(  a,1\right)  $ and $\nu\left(  a^{\prime},b^{\prime}\right)  $. So, we obtain%
\[
\psi^{8}=\psi_{15}+\psi_{26}.
\]
$\bullet$ Case  $a_{14}=0,\ a_{26}=0,\ a_{15}\neq0,C_{1}^{0}\neq1.~$we
use $\sigma\left(  b,2\right)  $ and $\nu\left(  a^{\prime},b^{\prime}\right)  $. So,  we obtain%
\[
\psi^{9}=\psi_{15}.
\]

\end{proof}


\section{Filiform Hom-Lie algebras of dimension $7$}

Any 7-dimensional filiform Hom-Lie algebra bracket  is given by  $\mu=\mu_{0}+\psi
,\ $where $\mu_{0}$ is the bracket of $L_{6}$ is defined in \eqref{Ln} and 
\[
\psi=a_{1,4}\psi_{1,4}+a_{1,5}\psi_{1,5}+a_{1,6}\psi_{1,6}\newline+a_{2,6}%
\psi_{2,6}.
\]
In this case, $\delta^{2}\psi=0$  is
equivalent to
\[
\rho_{33}=\rho_{11}C_{1}^{1},\ \rho_{00}=\rho_{11}C_{1}^{0}%
,\ \rho_{00}C_{1}^{0}=\rho_{11}C_{2}^{0},
\]
where $\rho_{ij}$ are the entries of the lower triangular matrix corresponding to  the structure map $\alpha$.\\
This cocycle condition may be written as
\[
\rho_{33}=\rho_{11}C_{1}^{1};\ \rho_{00}=\rho_{11}C_{1}^{0}%
;\ \rho_{11}((C_{1}^{0})^{2}-C_{2}^{0})=0.
\]
Therefore, we have the following solutions $\rho_{11}=0$ or $C_{2}^{0}=(C_{1}^{0})^{2}$, together
with $\rho_{33}=\rho_{11}C_{1}^{1}$ and $\rho_{00}=\rho_{11}C_{1}^{0}$.

The condition $\psi\circ\psi=0$ is equivalent to $\rho_{01}a_{14}\left(
a_{14}C_{2}^{0}-a_{26}C_{1}^{1}\right)  =0$, which  leads to these cases :
\begin{enumerate}
\item $a_{14}=0$.

\item $\rho_{01}=0$.

\item $a_{14}C_{2}^{0}-a_{26}C_{1}^{1}=0,$ which is equivalent to $C_{2}%
^{0}=\frac{a_{26}C_{1}^{1}}{a_{14}}$ with $a_{14}\neq0$.
\end{enumerate}

\subsection{ Basis changes and elimination of parameters}

\ 

A 7-dimensional filiform Hom-Lie algebra bracket $\mu\ $ may be written as
\[
\mu_{0}+\psi=\mu_{0}+a_{1,4}\psi_{1,4}+a_{1,5}\psi_{1,5}+a_{1,6}\psi
_{1,6}+a_{2,6}\psi_{2,6}.
\]

The use of the basis change $\sigma\left(  b,2\right)  $ gives a  new bracket with $\psi=b_{1,4}\psi_{1,4}+b_{1,5}\psi_{1,5}+b_{1,6}\psi_{1,6}+b_{2,6}
\psi_{2,6}$. The coefficients 
$B=$ $\{b_{k,m}\}$ satisfy
\begin{align*}
b_{14} &  =a_{14},\ b_{26}=a_{26},\\
b_{15} &  =a_{15}+ba_{14}\left(  C_{1}^{0}-1\right)  =:P_{1}(b),\\
b_{16} &  =a_{14}b^{2}+a_{16}+ba_{15}C_{1}^{0}+b^{2}a_{26}-ba_{15}-b^{2}%
a_{14}C_{1}^{0}=:Q_{1}(b).
\end{align*}
The use of the basis change $\sigma\left(  b,3\right)  $ gives a new bracket with coefficients 
$B=$ $\{b_{k,m}\}$  such that 
\begin{align*}
b_{14} &  =a_{14},\ b_{26}=a_{26},\\
b_{15} &  =a_{15},\\
b_{16} &  =-ba_{14}+a_{16}-ba_{26}\left(  1+C_{1}^{1}\right)  =:Q_{2}(b).
\end{align*}
The use of the basis change $\tau\left(  a,1\right)  $ gives a new bracket  with coefficients 
$B=$ $\{b_{k,m}\}$  such that 
\begin{align*}
b_{14} &  =a_{14},\\
b_{15} &  =a_{15}-aa_{14}^{2}\left(  1+C_{1}^{0}\right)  =:P_{2}(a),\\
b_{16} &  =a_{16}+a_{14}^{2}a\left(  C_{1}^{0}+1\right)  \left[  \left(
C_{1}^{0}+1\right)  a_{14}a_{14}C_{2}^{0}-a_{26}C_{1}^{1}\right]  \\
&  -a_{14}a\left[  \left(  C_{1}^{0}+1\right)  a_{15}+aC_{2}^{0}a_{14}%
^{2}-aa_{14}a_{26}C_{1}^{1}\right]  \\
&  -aa_{15}\left[  \left(  C_{1}^{0}+1\right)  a_{14}+a_{14}C_{2}^{0}%
-a_{26}C_{1}^{1}\right]  =:Q_{3}(a),\\
b_{26} &  =a_{26}.%
\end{align*}
The use  of the basis change $\tau\left(  a,2\right)  $ gives a new
bracket   with coefficients  $B=$ $\{b_{k,m}\}$  such that 
\begin{align*}
b_{14} &  =a_{14},\ 
b_{15}   =a_{15},\\
b_{16} &  =a_{16}+aa_{14}\left(  1+a_{26}C_{1}^{1}\right)  -a~a_{14}^{2}%
C_{2}^{0}-a\ a_{14}~a_{26}=:Q_{4}(a),\\
b_{26} &  =a_{26}.%
\end{align*}
The use of the basis change $\nu\left(  a,b\right)  $ gives a new bracket   with coefficients 
$B=$ $\{b_{k,m}\}\ $ such that 
\[
b_{k,m}=b\frac{a_{k,m}}{a^{m-2k}}.%
\]

One has also to write, with respect to the new basis, the corresponding matrix structure map $\rho$.
\begin{remark}\
\begin{enumerate}
\item In the previous basis changes, one may reduce $a_{15}$ and/or $a_{16}$
to zero if the corresponding polynomials $P_s$ or $Q_r$, where $s=1,2$ and $r=1,2,3,4$, admit a root.

\item Using the basis change $\nu\left(  a,b\right)  $ may transform two
non-zero coefficients $a_{k,m}$ and $a_{r,t}$ with $m-2k\neq t-2r$ into $1.$

\item If all the non-zero elements of $A=\{a_{i,j}\}_{(i,j)\in\triangle}$ have
the same homogeneity, i.e. $m-2k-1=t-2r-1$ if $a_{k,m}$ $\neq0\ $and
$a_{r,t}\neq0$, then we may transform only one coefficient of $ A$ into $1$.
\end{enumerate}
\end{remark}

\subsection{Classification}


\begin{theorem}[7-dimensional Classification]
Every $7$-dimensional filiform Hom-Lie algebra is isomorphic to one of the
following pairwise non isomorphic Hom-Lie algebra $(\mathfrak{g},[\cdot
,\cdot], \alpha)$, where the bracket is defined by the multiplication $\mu
_{m}^{i}=\mu_{0}+\psi^{i}\ $ and the linear map $\alpha$ is given by its
corresponding matrix with respect to the same basis:

\begin{itemize}

\item[$\blacktriangleright$]  $\mu_{7}^{1}:\ \mu_{0},$ with $\rho$ given by any lower triangular matrix.

\item[$\blacktriangleright$]  $\mu_{7}^{2}:\ \mu_{0}+\psi_{14}+\beta\psi_{26};\ $with $\beta\neq0$ and the
linear map $\rho$ given by matrices of the form
\[
\left(
\begin{array}
[c]{ccccccc}%
C_{1}^{0}\rho_{11} & 0 & 0 & 0 & 0 & 0 & 0\\
\rho_{01} & \rho_{11} & 0 & 0 & 0 & 0 & 0\\
\rho_{02} & \rho_{12} & \rho_{22} & 0 & 0 & 0 & 0\\
\rho_{03} & \rho_{13} & \frac{-\left(  \beta C_{1}^{1}-C_{2}^{0}\right)
\rho_{01}}{\beta}-\rho_{12} & \frac{\left(  \beta C_{1}^{1}-C_{2}^{0}%
+(C_{1}^{0})^{2}\right)  \rho_{11}}{\beta} & 0 & 0 & 0\\
\rho_{04} & \rho_{14} & \rho_{24} & \rho_{34} & \rho_{44} & 0 & 0\\
\rho_{05} & \rho_{15} & \rho_{25} & \rho_{35} & \rho_{45} & \rho_{55} & 0\\
\rho_{06} & \rho_{16} & \rho_{26} & \rho_{36} & \rho_{46} & \rho_{56} &
\rho_{66}%
\end{array}
\right).
\]

\item[$\blacktriangleright$] $\mu_{7}^{3}:\ \mu_{0}+\psi_{26};\ $with the linear map $\rho$ given by matrices of the form
\[
\left(
\begin{array}
[c]{ccccccc}%
\rho_{00} & 0 & 0 & 0 & 0 & 0 & 0\\
\rho_{01} & \rho_{11} & 0 & 0 & 0 & 0 & 0\\
\rho_{02} & \rho_{12} & \rho_{22} & 0 & 0 & 0 & 0\\
\rho_{03} & \rho_{13} & \rho_{23} & C_{1}^{1}\rho_{11} & 0 & 0 & 0\\
\rho_{04} & \rho_{14} & \rho_{34} & \rho_{34} & \rho_{44} & 0 & 0\\
\rho_{05} & \rho_{15} & \rho_{35} & \rho_{35} & \rho_{45} & \rho_{55} & 0\\
\rho_{06} & \rho_{16} & \rho_{36} & \rho_{36} & \rho_{46} & \rho_{56} &
\rho_{66}%
\end{array}
\right).
\]

\item[$\blacktriangleright$] $\mu_{7}^{4}:\ \mu_{0}+\psi_{15}+\psi_{26};\ $with the linear map $\rho$
given by matrices of the form
\[
\left(
\begin{array}
[c]{ccccccc}%
\rho_{00} & 0 & 0 & 0 & 0 & 0 & 0\\
\rho_{01} & \rho_{11} & 0 & 0 & 0 & 0 & 0\\
\rho_{02} & \rho_{12} & \rho_{22} & 0 & 0 & 0 & 0\\
\rho_{03} & \rho_{13} & \rho_{00}+C_{1}^{1}\rho_{11}-\rho_{12} & C_{1}^{1}%
\rho_{11} & 0 & 0 & 0\\
\rho_{04} & \rho_{14} & \rho_{34} & \rho_{34} & \rho_{44} & 0 & 0\\
\rho_{05} & \rho_{15} & \rho_{35} & \rho_{35} & \rho_{45} & \rho_{55} & 0\\
\rho_{06} & \rho_{16} & \rho_{36} & \rho_{36} & \rho_{46} & \rho_{56} &
\rho_{66}%
\end{array}
\right).
\]

\item[$\blacktriangleright$] $\mu_{7}^{5}:\ \mu_{0}+\psi_{15};\ $with the linear map $\rho$ given by matrices of the form
\[
\left(
\begin{array}
[c]{ccccccc}%
C_{1}^{0}\rho_{11} & 0 & 0 & 0 & 0 & 0 & 0\\
\rho_{01} & \rho_{11} & 0 & 0 & 0 & 0 & 0\\
\rho_{02} & \rho_{12} & \rho_{22} & 0 & 0 & 0 & 0\\
\rho_{03} & \rho_{13} & \rho_{23} & \rho_{33} & 0 & 0 & 0\\
\rho_{04} & \rho_{14} & \rho_{34} & \rho_{34} & \rho_{44} & 0 & 0\\
\rho_{05} & \rho_{15} & \rho_{35} & \rho_{35} & \rho_{45} & \rho_{55} & 0\\
\rho_{06} & \rho_{16} & \rho_{36} & \rho_{36} & \rho_{46} & \rho_{56} &
\rho_{66}%
\end{array}
\right).
\]

\item[$\blacktriangleright$] $\mu_{7}^{6}:\mu_{0}+\psi_{16};\ $ with the linear map $\rho$ given by any
lower triangular matrix.

\item[$\blacktriangleright$] $\mu_{7}^{7}:\ \mu_{0}+\psi_{14}+\psi_{16};\ $with the linear map $\rho$
given by matrices of the form
\[
\left(
\begin{array}
[c]{ccccccc}%
0 & 0 & 0 & 0 & 0 & 0 & 0\\
0 & 0 & 0 & 0 & 0 & 0 & 0\\
\rho_{02} & \rho_{12} & \rho_{22} & 0 & 0 & 0 & 0\\
\rho_{03} & \rho_{13} & \rho_{23} & \rho_{33} & 0 & 0 & 0\\
\rho_{04} & \rho_{14} & \rho_{34} & \rho_{34} & \rho_{44} & 0 & 0\\
\rho_{05} & \rho_{15} & \rho_{35} & \rho_{35} & \rho_{45} & \rho_{55} & 0\\
\rho_{06} & \rho_{16} & \rho_{36} & \rho_{36} & \rho_{46} & \rho_{56} &
\rho_{66}%
\end{array}
\right).
\]

\item[$\blacktriangleright$] $\mu_{7}^{8}:\ \mu_{0}+\psi_{14};\ $with the linear map $\rho$ given by matrices of the form
\[
\left(
\begin{array}
[c]{ccccccc}%
0 & 0 & 0 & 0 & 0 & 0 & 0\\
0 & 0 & 0 & 0 & 0 & 0 & 0\\
\rho_{02} & \rho_{12} & \rho_{22} & 0 & 0 & 0 & 0\\
\rho_{03} & \rho_{13} & \rho_{23} & \rho_{33} & 0 & 0 & 0\\
\rho_{04} & \rho_{14} & \rho_{34} & \rho_{34} & \rho_{44} & 0 & 0\\
\rho_{05} & \rho_{15} & \rho_{35} & \rho_{35} & \rho_{45} & \rho_{55} & 0\\
\rho_{06} & \rho_{16} & \rho_{36} & \rho_{36} & \rho_{46} & \rho_{56} &
\rho_{66}%
\end{array}
\right).
\]
\end{itemize}
\end{theorem}

\begin{proof}
[ Sketch of proof]We study the following cases. \\
$\bullet$ Case $a_{14}\neq0,\ a_{26}\neq0\ $and$~C_{1}^{0}%
\neq1,$ $\rho_{01}=0,\ a_{14}^{2}-a_{14}^{2}C_{2}^{0}+a_{14}C_{1}^{1}%
a_{26}-a_{14}a_{26}\neq0.$\newline We use $\sigma\left(  b,2\right)  $ and
then $\tau\left(  a,2\right)  $ $\ $and $\nu\left(  a^{{\prime}},b^{\prime
}\right)  $. So,  we obtain%
\[
\psi^{1}=\psi_{14}+\frac{a_{26}}{a_{14}^{2}}\psi_{26}.%
\]
$\bullet$  Case $~a_{14}\neq0,\ a_{26}\neq0,\ C_{1}^{0}=1,$ $\rho_{01}=0.\ $We use
$\tau\left(  a,1\right)  $, then $\sigma\left(  b,2\right)  $ and $\nu\left(
a^{{\prime}},b^{\prime}\right)  $. So, we obtain%
\[
\psi^{2}=\psi_{14}+\frac{a_{26}}{a_{14}^{2}}\psi_{26}.%
\]
$\bullet$  Case $a_{14}\neq0,\ a_{26}\neq0,\ C_{1}^{0}=1,$ $C_{2}^{0}%
a_{14}=a_{26}C_{1}^{1}.\ $We use $\tau\left(  a,1\right)  $, then
$\sigma\left(  b,2\right)  $ and $\nu\left(  a^{\prime},b^{\prime
}\right)  $. So,  we obtain%
\[
\psi^{3}=\psi_{14}+\frac{a_{26}}{a_{14}^{2}}\psi_{26}.%
\]
$\bullet$  Case  $a_{14}\neq0,a_{26}\neq0\ $and$~C_{1}^{0}\neq1,$ $\rho
_{01}=0,\ a_{14}^{2}-a_{14}^{2}C_{2}^{0}+a_{14}C_{1}^{1}a_{26}-a_{14}a_{26}%
=0$. We use $\sigma\left(  b_{1},2\right)  $ and then $\sigma\left(
b_{2},3\right)  $ and $\nu\left(  a^{\prime},b^{\prime}\right) $. So, 
we obtain%
\[
\psi^{4}=\psi_{14}+\frac{a_{26}}{a_{14}^{2}}\psi_{26}.
\]
$\bullet$  Case  $a_{14}\neq0,a_{26}\neq0\ $and$~C_{1}^{0}\neq1,$ $a_{14}%
C_{2}^{0}=C_{1}^{1}a_{26},\ a_{14}^{2}-a_{14}^{2}C_{2}^{0}+a_{14}C_{1}%
^{1}a_{26}-a_{14}a_{26}=0$.\ We use $\sigma\left(  b_{1},2\right)  $ and then
$\sigma\left(  b_{2},3\right)  $ $\ $and $\nu\left(  a^{\prime}%
,b^{\prime}\right)  $. So, we obtain%
\[
\psi^{5}=\psi_{14}+\frac{a_{26}}{a_{14}^{2}}\psi_{26}.%
\]
$\bullet$  Case  $a_{14}=0,a_{26}\neq0,a_{15}=0,a_{16}\neq0.$ We use
$\sigma\left(  b,2\right)  $ and $\nu\left(  a^{\prime},b^{\prime
}\right)  $. So, we obtain%
\[
\psi^{6}=\psi_{26}.%
\]
$\bullet$  Case  $a_{14}=0,\ a_{26}\neq0,a_{15}\neq0,\ C_{1}^{1}\neq-1.\ $ We
use $\sigma\left(  b,3\right)  $ and $\nu\left(  a^{{\prime}},b^{{\prime}%
}\right)  $. So, we obtain%
\[
\psi^{7}=\psi_{15}+\psi_{26}.%
\]
$\bullet$  Case  $a_{14}=0,\ a_{26}\neq0,a_{15}\neq0,\ C_{1}^{1}=-1.\ $We
use $\tau\left(  a,1\right)  $ and $\nu\left(  a^{{\prime}},b^{{\prime}%
}\right)  $. So, we obtain%
\[
\psi^{8}=\psi_{15}+\psi_{26}.%
\]
$\bullet$  Case  $a_{14}=0,\ a_{26}=0,\ a_{15}\neq0,C_{1}^{0}\neq1.~$ We
use $\sigma\left(  b,2\right)  $ and $\nu\left(  a^{{\prime}},b^{{\prime}%
}\right)  $. So, we obtain%
\[
\psi^{9}=\psi_{15}.%
\]
$\bullet$  Case  $a_{14}=0,\ a_{26}=0,\ a_{15}\neq0,\ a_{16}\neq
0,\ C_{1}^{0}=1.\ $We use $\nu\left(  a,b\right)  $. So, we obtain
\[
\psi^{10}=\psi_{15}+\psi_{16}.%
\]
$\bullet$  Case  $a_{14}=0,\ a_{26}=0,\ a_{15}=0,$ $a_{16}\neq0.\ $We use
$\nu\left(  a,b\right)  $. So,  we obtain
\[
\psi^{11}=\psi_{16}.%
\]
$\bullet$  Case  $a_{14}\neq0,\ a_{26}=0,$ $C_{1}^{0}=1,\ C_{2}^{0}%
=1.\ $We use $\tau\left(  a,1\right)  $ then $\nu\left(  a^{{\prime}%
},b\right)  $. So, we get
\[
\psi^{12}=\psi_{14}+\psi_{16}.%
\]
Let us notice  that $\psi=\psi_{14}+\psi_{15}$ is isomorphic to $\psi=\psi_{14}+\psi_{16}%
$.\ One may use $\tau\left(  a^{\ast
},1\right)  $ and $\nu\left(  a^{\ast\ast},b^{\ast}\right)  .$ \newline%
$\bullet$  Case  $a_{14}\neq0,\ a_{26}=0,$ $C_{1}^{0}=1,\ C_{2}^{0}%
\neq1,\ \rho_{00}=0.\ $We use $\tau\left(  a,1\right)  $, then $\sigma\left(
b,3\right)  $ and $\nu\left(  a^{{\prime}},b^{{\prime}}\right)  $. So, we obtain%
\[
\psi^{13}=\psi_{14}.
\]
\newline$\bullet$  Case  $a_{14}\neq0,\ a_{26}=0,\ C_{2}^{0}\neq
1,\ C_{1}^{0}\neq1.$ We use $\sigma\left(  b,2\right)  \ $then $\sigma\left(
b^{{\prime}},3\right)  $ and $\nu\left(  a^{{\prime}},b"\right)  $. So,  we
obtain%
\[
\psi^{14}=\psi_{14}.
\]
$\bullet$  Case $a_{14}\neq0,\ a_{26}=0,\ C_{2}^{0}=1,\ C_{1}^{0}\neq
1,$\ $\rho_{00}=0,\ a_{16}\neq0.$ We use $\sigma\left(  b,2\right)  \ $then
$\nu\left(  a,b^{{\prime}}\right)  $. So, we obtain%
\[
\psi^{15}=\psi_{14}+\psi_{16}.
\]
$\bullet$  Case  $a_{14}\neq0,\ a_{26}=0,\ C_{2}^{0}=1,\ C_{1}^{0}\neq
1,$\ $\rho_{00}=0,\ a_{16}=0.\ $We use $\sigma\left(  b,2\right)  \ $then
$\nu\left(  a,b^{{\prime}}\right)  $. So, we obtain
\[
\psi^{16}=\psi_{14}.
\]

As we can see for example in the list of $\psi^{i}$ here above, we found five
times $\psi_{14}+\frac{a_{26}}{a_{14}^{2}}\psi_{26}$ after using different
basis changes.  These brackets are written in the
same form but they are not isomorphic, the difference is in the
parameters $C_{j}^{i}$. The corresponding Hom-Lie algebras are non-isomorphic.
\end{proof}

\section{Multiplicative filiform Hom-Lie algebras}

In this section, we reconsider the previous classification and provide the
matrices of the structure maps in order to obtain multiplicative filiform
Hom-Lie algebras. We do not give the classification in dimension 7 because
there is a big number of cases and the expression of the matrices is too
complicate.\newline

\noindent\textbf{Dimension 3}:

$\blacktriangleright$ $\mu_{3}^{1}:\mu_{0},$ with a structure map $\rho$ given by lower triangular
matrices as

$
\left(
\begin{array}
[c]{ccc}%
\rho_{00} & 0 & 0\\
\rho_{01} & \rho_{11} & 0\\
\rho_{02} & \rho_{12} & \rho_{00}\rho_{11}%
\end{array}
\right) 
$
 or  
 $
\left(
\begin{array}
[c]{ccc}
0 & 0 & 0\\
\rho_{01} & \rho_{11} & 0\\
\rho_{02} & \rho_{12} & 0%
\end{array}
\right) .
$
\\
\textbf{Dimension 4}:

$\blacktriangleright$ $\mu_{4}^{1}:\ \mu_{0},\ $ with a structure map $\rho$ given by

$
\left(
\begin{array}
[c]{cccc}%
\rho_{00} & 0 & 0 & 0\\
\rho_{01} & \frac{\rho_{22}}{\rho_{00}} & 0 & 0\\
\rho_{02} & \rho_{12} & \rho_{22} & 0\\
\rho_{03} & \rho_{13} & \rho_{12}\rho_{00} & \rho_{00}\rho_{22}%
\end{array}
\right) 
$
or 
$
\left(
\begin{array}
[c]{cccc}%
0 & 0 & 0 & 0\\
\rho_{01} & \rho_{11} & 0 & 0\\
\rho_{02} & \rho_{12} & 0 & 0\\
\rho_{03} & \rho_{13} & 0& 0
\end{array}
\right).
$
\\
\textbf{Dimension 5}:

$\blacktriangleright$ $\mu_{5}^{1}:\mu_{0},$ with a structure map $\rho$ given by
\[
\left(
\begin{array}
[c]{ccccc}%
\rho_{00} & 0 & 0 & 0 & 0\\
\rho_{01} & \frac{\rho_{22}}{\rho_{00}} & 0 & 0 & 0\\
\rho_{02} & \frac{\rho_{23}}{\rho_{00}} & \rho_{22} & 0 & 0\\
\rho_{03} & \rho_{13} & \rho_{23} & \rho_{00}\rho_{22} & 0\\
\rho_{04} & \rho_{14} & \rho_{00}\rho_{13} & \rho_{00}\rho_{23} & \rho
_{00}^{2}\rho_{22}%
\end{array}
\right)
\text{ or }
\left(
\begin{array}
[c]{ccccc}%
0 & 0 & 0 & 0 & 0\\
\rho_{01} & \rho_{11} & 0 & 0 & 0\\
\rho_{02} & \rho_{12} & 0 & 0 & 0\\
\rho_{03} & \rho_{13} & 0 & 0 & 0\\
\rho_{04} & \rho_{14} & 0 & 0 & 0
\end{array}
\right) .
\]

$\blacktriangleright$ $\mu_{5}^{2}:\mu_{0}+\psi_{1,4}$ with a structure map $\rho$ given by%
\[
\left(
\begin{array}
[c]{ccccc}%
0 & 0 & 0 & 0 & 0\\
\rho_{01} & \rho_{11} & 0 & 0 & 0\\
\rho_{02} & \rho_{12} & 0 & 0 & 0\\
\rho_{03} & \rho_{13} & 0 & 0 & 0\\
\rho_{04} & \rho_{14} & -\rho_{11}\rho_{02}+\rho_{12}\rho_{01} & 0 & 0
\end{array}
\right)
\]
or
\[
\left(
\begin{array}
[c]{ccccc}%
\rho_{00} & 0 & 0 & 0 & 0\\
\rho_{01} & 0 & 0 & 0 & 0\\
\rho_{02} & \rho_{12} & 0 & 0 & 0\\
\rho_{03} & \rho_{13} & \rho_{00}\rho_{12} & 0 & 0\\
\rho_{04} & \rho_{14} & \rho_{01}\rho_{12}+\rho_{00}\rho_{13} & \rho_{00}%
^{2}\rho_{12} & 0
\end{array}
\right)
\]
or
\[
\left(
\begin{array}
[c]{ccccc}%
\rho_{00} & 0 & 0 & 0 & 0\\
\rho_{01} & \rho_{00}^{2} & 0 & 0 & 0\\
\rho_{02} & \rho_{12} & \rho_{00}^{3} & 0 & 0\\
\rho_{03} & \rho_{13} & \rho_{00}\rho_{12} & \rho_{00}^{4} & 0\\
\rho_{04} & \rho_{14} & \rho_{01}\rho_{12}+\rho_{00}\rho_{13}-\rho_{00}%
^{2}\rho_{02} & \rho_{00}^{2}(\rho_{00}+\rho_{01}+\rho_{12}) & \rho_{00}^{5}%
\end{array}
\right)  .
\]
\\
\textbf{Dimension 6}

$\blacktriangleright$ $\mu_{6}^{1}:\ \mu_{0}~,\ $with a structure map $\rho$ given by a matrix of
the form%
\[
\left(
\begin{array}
[c]{cccccc}%
\rho_{00} & 0 & 0 & 0 & 0 & 0\\
\rho_{01} & \frac{\rho_{22}}{\rho_{00}} & 0 & 0 & 0 & 0\\
\rho_{02} & \frac{\rho_{23}}{\rho_{00}} & \rho_{22} & 0 & 0 & 0\\
\rho_{03} & \frac{\rho_{24}}{\rho_{00}} & \rho_{23} & \rho_{00}\rho_{2} & 0 &
0\\
\rho_{04} & \rho_{14} & \rho_{24} & \rho_{23}\rho_{00} & \rho_{22} \rho
_{00}^{2} & 0\\
\rho_{05} & \rho_{15} & \rho_{14}\rho_{00} & \rho_{24}\rho_{00} & \rho
_{23}\rho_{00}^{2} & \rho_{00}^{3}\rho_{22}%
\end{array}
\right)
\]
or
\[
\left(
\begin{array}
[c]{cccccc}%
0 & 0 & 0 & 0 & 0 & 0\\
\rho_{01} & \rho_{11} & 0 & 0 & 0 & 0\\
\rho_{02} & \rho_{12} & 0 & 0 & 0 & 0\\
\rho_{03} & \rho_{31} & 0 & 0 & 0 & 0\\
\rho_{04} & \rho_{14} & 0 & 0 & 0 & 0\\
\rho_{05} & \rho_{15} & 0 & 0 & 0 & 0
\end{array}
\right).
\]

$\blacktriangleright$ $\mu_{6}^{2}:\ \mu_{0}+\psi_{14}+\psi_{25}$ with a structure map $\rho$ given
by a matrix of the form
\[
\left(
\begin{array}
[c]{cccccc}%
0 & 0 & 0 & 0 & 0 & 0\\
\rho_{01} & \rho_{11} & 0 & 0 & 0 & 0\\
\rho_{02} & -C_{1}^{0}\rho_{11} & 0 & 0 & 0 & 0\\
\rho_{03} & \rho_{13} & 0 & 0 & 0 & 0\\
\rho_{04} & \rho_{14} & -\rho_{11}(C_{1}^{0}\rho_{01}+\rho_{02}) & 0 & 0 & 0\\
\rho_{05} & \rho_{15} & \rho_{13}(C_{1}^{0}\rho_{01}+\rho_{02}) & 0 & 0 & 0
\end{array}
\right)
\]
or
\[
\left(
\begin{array}
[c]{cccccc}%
-\frac{1}{2}(1+\sqrt{5}) & 0 & 0 & 0 & 0 & 0\\
\rho_{01} & \rho_{11} & 0 & 0 & 0 & 0\\
\rho_{02} & \frac{1}{2}(3+\sqrt{5}) & 0 & 0 & 0 & 0\\
\rho_{03} & \rho_{13} & -2-\sqrt{5} & 0 & 0 & 0\\
\rho_{04} & \rho_{14} & \frac{1}{2}((3+\sqrt{5})\rho_{01}-(1+\sqrt{5}%
)\rho_{31}) & \frac{1}{2}(7+3\sqrt{5}) & 0 & 0\\
\rho_{05} & \rho_{15} & \rho_{25} & \rho_{35} & 0 & 0
\end{array}
\right)
\]
with
\begin{align*}
\rho_{35}  &  =\frac{1}{2}((3+\sqrt{5})\rho_{13}-(4+2\sqrt{5})\rho
_{02}-(4+2\sqrt{5})(C_{1}^{0}+1)\rho_{01}),\\
\rho_{25}  &  =\frac{1}{2}(-(3+\sqrt{5})\rho_{03}-2(C_{1}^{0}\rho_{01}%
+\rho_{02})\rho_{31}-(1+\sqrt{5})\rho_{14}),
\end{align*}
or%
\[
\left(
\begin{array}
[c]{cccccc}%
-\frac{1}{2}(1-\sqrt{5}) & 0 & 0 & 0 & 0 & 0\\
\rho_{01} & \rho_{11} & 0 & 0 & 0 & 0\\
\rho_{02} & \frac{1}{2}(3-\sqrt{5}) & 0 & 0 & 0 & 0\\
\rho_{03} & \rho_{13} & -2+\sqrt{5} & 0 & 0 & 0\\
\rho_{04} & \rho_{14} & \frac{1}{2}((3-\sqrt{5})\rho_{01}-(1-\sqrt{5}%
)\rho_{31}) & \frac{1}{2}(7-3\sqrt{5}) & 0 & 0\\
\rho_{05} & \rho_{15} & \rho_{25} & \rho_{35} & 0 & 0
\end{array}
\right)
\]
with
\begin{align*}
\rho_{35}  &  =\frac{1}{2}((3-\sqrt{5})\rho_{13}-(4-2\sqrt{5})\rho
_{02}-(4-2\sqrt{5})(C_{1}^{0}+1)\rho_{01}),\\
\rho_{25}  &  =\frac{1}{2}(-(3-\sqrt{5})\rho_{03}-2(C_{1}^{0}\rho_{01}%
+\rho_{02})\rho_{31}-(1-\sqrt{5})\rho_{14}).
\end{align*}

$\blacktriangleright$ $\mu_{6}^{3}:\ \mu_{0}+\psi_{14}$, with a structure map $\rho$ given by a
matrix of the form
\[
\left(
\begin{array}
[c]{cccccc}%
0 & 0 & 0 & 0 & 0 & 0\\
\rho_{01} & 0 & 0 & 0 & 0 & 0\\
\rho_{02} & \rho_{12} & 0 & 0 & 0 & 0\\
\rho_{03} & \rho_{13} & 0 & 0 & 0 & 0\\
\rho_{04} & \rho_{14} & \rho_{01}\rho_{12} & 0 & 0 & 0\\
\rho_{05} & \rho_{15} & C_{1}^{0}\rho_{01}\rho_{13} & 0 & 0 & 0
\end{array}
\right)
\]
or
\[
\left(
\begin{array}
[c]{cccccc}%
\frac{1}{C_{1}^{0}} & 0 & 0 & 0 & 0 & 0\\
\rho_{01} & \frac{1}{(C_{1}^{0})^{2}} & 0 & 0 & 0 & 0\\
\rho_{02} & \frac{(1+C_{1}^{0})\rho_{01}}{(-1+C_{1}^{0})C_{1}^{0}} & \frac
{1}{(C_{1}^{0})^{3}} & 0 & 0 & 0\\
\rho_{03} & \rho_{13} & \frac{(1+C_{1}^{0})\rho_{01}}{(-1+C_{1}^{0})(C_{1}%
^{0})^{2}} & \frac{1}{(C_{1}^{0})^{4}} & 0 & 0\\
\rho_{04} & \rho_{14} & \rho_{24} & \frac{2\rho_{01}}{(-1+C_{1}^{0})(C_{1}%
^{0})^{2}} & \frac{1}{(C_{1}^{0})^{5}} & 0\\
\rho_{05} & \rho_{15} & \rho_{25} & \rho_{35} & \frac{(1+C_{1}^{0})\rho_{01}%
}{(-1+C_{1}^{0})(C_{1}^{0})^{2}} & \frac{1}{(C_{1}^{0})^{6}}%
\end{array}
\right)
\]
with
\begin{align*}
\rho_{35} &  =\left(  (1+C_{1}^{0})\rho_{01}^{2}+(-1+C_{1}^{0})\rho
_{24}\right)  /(-1+C_{1}^{0})(C_{1}^{0}),\\
\rho_{25} &  =\left(  -\rho_{03}+\rho_{14}+\rho_{13}\rho_{01}(C_{1}^{0}%
)^{2}\right)  /C_{1}^{0},\\
\rho_{02} &  =\frac{(1+C_{1}^{0})\rho_{01}^{2}+((C_{1}^{0})^{2}-(C_{1}%
^{0})^{3})\rho_{24}+C_{1}^{0}\rho_{01}^{2}+(C_{1}^{0})^{2}\rho_{01}^{2}%
-\rho_{13}C_{1}^{0}}{(-1+C_{1}^{0})}.
\end{align*}

$\blacktriangleright$ $\mu_{6}^{4}:\ \mu_{0}+\psi_{1,5},$ with a structure map $\rho$ given by a
matrix of the form
\[
\left(
\begin{array}
[c]{cccccc}%
\rho_{00} & 0 & 0 & 0 & 0 & 0\\
\rho_{01} & 0 & 0 & 0 & 0 & 0\\
\rho_{02} & \rho_{23} & 0 & 0 & 0 & 0\\
\rho_{03} & \frac{\rho_{24}}{\rho_{00}} & \rho_{00}\rho_{12} & 0 & 0 & 0\\
\rho_{04} & \rho_{14} & \rho_{24} & \rho_{12}\rho_{00}^{2} & 0 & 0\\
\rho_{05} & \rho_{15} & \rho_{12}\rho_{01}+\rho_{14}\rho_{00} & \rho_{24}%
\rho_{00} & \rho_{12}\rho_{00}^{3} & 0
\end{array}
\right)
\]
or
\[
\left(
\begin{array}
[c]{cccccc}%
0 & 0 & 0 & 0 & 0 & 0\\
\rho_{01} & \rho_{11} & 0 & 0 & 0 & 0\\
\rho_{02} & \rho_{12} & 0 & 0 & 0 & 0\\
\rho_{03} & \rho_{31} & 0 & 0 & 0 & 0\\
\rho_{04} & \rho_{14} & 0 & 0 & 0 & 0\\
\rho_{05} & \rho_{15} & -\rho_{11}\rho_{02}+\rho_{01}\rho_{12} & 0 & 0 & 0
\end{array}
\right)
\]
or%
\[
\left(
\begin{array}
[c]{cccccc}%
\rho_{00} & 0 & 0 & 0 & 0 & 0\\
\rho_{01} & \rho_{00}^{3} & 0 & 0 & 0 & 0\\
\rho_{02} & \rho_{12} & \rho_{00}^{4} & 0 & 0 & 0\\
\rho_{03} & \frac{\rho_{24}}{\rho_{00}} & \rho_{00}\rho_{12} & \rho_{00}^{5} &
0 & 0\\
\rho_{04} & \rho_{14} & \rho_{24} & \rho_{00}^{2}\rho_{12} & \rho_{00}^{6} &
0\\
\rho_{05} & \rho_{15} & -\rho_{00}^{3}\rho_{02}+\rho_{14}\rho_{00}+\rho
_{01}\rho_{12} & \rho_{00}(\rho_{00}^{3}\rho_{01}+\rho_{24}) & \rho_{00}%
^{3}\rho_{12} & \rho_{00}^{7}%
\end{array}
\right)
\]
or
\[
\left(
\begin{array}
[c]{cccccc}%
\rho_{00} & 0 & 0 & 0 & 0 & 0\\
\rho_{01} & \rho_{00}^{3} & 0 & 0 & 0 & 0\\
\rho_{02} & \rho_{12} & \rho_{00}^{4} & 0 & 0 & 0\\
\rho_{03} & \frac{\rho_{24}}{\rho_{00}} & \rho_{00}\rho_{12} & \rho_{00}^{5} &
0 & 0\\
\rho_{04} & \rho_{14} & 0 & \rho_{00}^{2}\rho_{12} & \rho_{00}^{6} & 0\\
\rho_{05} & \rho_{15} & -\rho_{00}^{3}\rho_{02}+\rho_{14}\rho_{00}+\rho
_{01}\rho_{12} & \rho_{00}^{4}\rho_{01} & \rho_{00}^{3}\rho_{12} & \rho
_{00}^{7}%
\end{array}
\right)
\]
or
\[
\left(
\begin{array}
[c]{cccccc}%
\rho_{00} & 0 & 0 & 0 & 0 & 0\\
\rho_{01} & 0 & 0 & 0 & 0 & 0\\
\rho_{02} & \rho_{12} & 0 & 0 & 0 & 0\\
\rho_{03} & 0 & \rho_{00}\rho_{12} & 0 & 0 & 0\\
\rho_{04} & \rho_{14} & 0 & \rho_{00}^{2}\rho_{12} & 0 & 0\\
\rho_{05} & \rho_{15} & \rho_{14}\rho_{00}+\rho_{01}\rho_{12} & 0 & \rho
_{00}^{3}\rho_{12} & 0
\end{array}
\right) .
\]

\end{document}